\documentclass[times,sort&compress,3p]{elsarticle}
\usepackage{amssymb,amsmath, amsthm, mathabx,mathtools}
\usepackage{amsmath,amsfonts,amsthm,epsfig,epstopdf,url,array}
\usepackage{tikz}
\usetikzlibrary{shapes,snakes}
\usepackage{color, enumerate}
\usepackage[colorlinks=true,linkcolor=blue, citecolor=blue]{hyperref}
\usepackage{pgfplots}

\usepackage{bm}
\usepackage{graphicx}
\usepackage{caption}
\usepackage{subcaption}
\usepackage{titlesec}
\usepackage{epstopdf}

\usepackage{sectsty}
\usepackage{lipsum}
\usepackage{algorithm}
\usepackage[noend]{algpseudocode}
\usepackage{algpseudocode}
\usepackage{pifont}

\newcommand{\Gb}{\mathbf{G}}

\newcommand{\Hb}{\mathbf{H}}

\newcommand{\hlam}{\widehat{\lambda}}

\newcommand{\Ub}{\mathbf{U}}

\newcommand{\norm}[1]{\left\lVert#1\right\rVert}

\DeclareMathAccent{\wideparen}{\mathord}{largesymbols}{"F3}

\newcommand{\rO} {\mathrm{O}}
\newcommand{\ro} {\mathrm{o}}
\newcommand{\deq}{\mathrel{\mathop:}=}

\newcommand{\R} {\mathbb{R}}
\newcommand{\C} {\mathbb{C}}
\newcommand{\N} {\mathbb{N}}

\newcommand{\adj}{^*}

\newcommand{\dist} {\mathrm{dist}}
\newcommand{\CR} {\C\setminus\R_{+}}

\DeclareMathOperator{\diag}{diag}
\DeclareMathOperator{\tr}{tr}
\DeclareMathOperator{\Tr}{Tr}

\DeclareMathOperator{\supp}{supp}

\DeclareMathOperator{\re}{\mathrm{Re}}
\DeclareMathOperator{\im}{\mathrm{Im}}

\newcommand{\caC}{{\mathcal C}}
\newcommand{\caD}{{\mathcal D}}

\newcommand{\caG}{{\mathcal G}}
\newcommand{\caH}{{\mathcal H}}
\newcommand{\caI}{{\mathcal I}}

\newcommand{\caO}{{\mathcal O}}

\newcommand{\caQ}{{\mathcal Q}}

\newcommand{\caT}{{\mathcal T}}

\newcommand{\ub}{\mathbf{u}}
\newcommand{\vb}{\mathbf{v}}

\newcommand{\bsd}{{\boldsymbol d}}

\newcommand{\wt}{\widetilde}
\newcommand{\ol}{\overline}
\newcommand{\wh}{\widehat}

\newcommand{\beq}{ \begin{equation} }
\newcommand{\eeq}{ \end{equation} }
\newcommand{\beqs}{	\begin{equation*}	}
\newcommand{\eeqs}{	\end{equation*}	}

\newcommand{\dd}{\mathrm{d}}
\newcommand{\ii}{\mathrm{i}}

\newcommand\prob[1]{\mathbb{P}\left[#1\right]}

\newcommand\Absv[1]{\left\vert#1\right\vert}
\newcommand\absv[1]{\vert#1\vert}
\newcommand\llbra{\llbracket}
\newcommand\rrbra{\rrbracket}

\newcommand\AND{\quad\text{and}\quad}

\numberwithin{equation}{section}

\theoremstyle{plain}
\newtheorem{thm}{Theorem}[section]

\newtheorem{lem}[thm]{Lemma}

\newtheorem{assu}[thm]{Assumption}
\newtheorem{prop}[thm]{Proposition}
\newtheorem{defn}[thm]{Definition}

\theoremstyle{remark}
\newtheorem{rem}[thm]{Remark}

\begin{document}

\begin{frontmatter}

\title{Spiked multiplicative  random matrices and principal components}

\author[1]{Xiucai Ding \corref{mycorrespondingauthor}}
\author[2]{Hong Chang Ji}

\address[1]{Department of Statistics, University of California, Davis.}
\address[2]{Institute of Science and Technology Austria}

\cortext[mycorrespondingauthor]{Email address: \url{xcading@ucdavis.edu}}

\begin{abstract}
In this paper, we study the eigenvalues and eigenvectors of the spiked invariant multiplicative models when the randomness is from Haar matrices. We establish the limits of the outlier eigenvalues $\widehat{\lambda}_i$ and the generalized components ($\langle \vb, \widehat{\ub}_i \rangle$ for any deterministic vector $\vb$) of the outlier eigenvectors $\widehat{\ub}_i$ with optimal convergence rates. Moreover, we prove that the non-outlier eigenvalues stick with those of the unspiked matrices and the non-outlier eigenvectors are delocalized. The results also hold near the so-called BBP transition and for degenerate spikes. On one hand, our results can be regarded as a refinement of the counterparts of \cite{outliermodel} under additional regularity conditions. On the other hand, they can be viewed as an analog of \cite{DYaos} by replacing the random matrix with i.i.d. entries with Haar random matrix.
\end{abstract}

\begin{keyword} 
Free multiplication of random matrices \sep Spiked model  \sep  Principal components \sep Local laws
\MSC[2020] Primary 60B20 \sep Secondary 46L54 \sep 62H12
\end{keyword}

\end{frontmatter}

\normalcolor
\section{Introduction}
Finite rank deformed random matrices have found applications in many scientific endeavors. In these contexts, the low-rank part is usually regarded as the signal whereas the random matrix part can be viewed as the high dimensional noise. From an application viewpoint, researchers are interested in understanding the signal part from the noisy matrix especially from the first few largest eigenvalues and eigenvectors, which are closely related to the principal component analysis (PCA) \cite{jolliffe2013principal}. 

In the literature,  a popular and sophisticated model is the spiked covariance matrix model proposed by Johnstone \cite{10.1214/aos/1009210544}. In such a model, a finite number of spikes (eigenvalues detached from the bulk of the spectrum) are added to the spectrum of the population covariance matrix. Specifically, consider that 
\begin{equation*}
\widehat{Y}=\widehat{A}^{1/2} X.
\end{equation*}
Here $\widehat{A}$ is the spiked covaraince matrix constructed by adding a finite rank perturbation to some non-spiked positive definite matrix $A,$ and $X=(x_{ij})$ is the main random source where $x_{ij}$'s are i.i.d. centered random variables. An extension  is the spiked separable covariance matrix \cite{DYaos}, where the data matrix is 
$Y=\widehat{A}^{1/2}X\widehat{B}^{1/2},$
with another spiked matrix $\widehat{B}$. In spatiotemporal data analysis, $\widehat{A}$ and $\widehat{B}$ are respectively the spatial and
temporal covariance matrices \cite{PAUL200937}.








While the assumption that $X$ has i.i.d. entries has been useful in many instances, other types of random matrices also appear naturally in certain applications. An important example is the Haar  distributed random matrices which have been used in statistical learning theory, for instance, see \cite{ELnips,MR4130669,nips20201,Liu2020Ridge,2020arXiv200500511Y}. In the current paper, we aim to study the spiked random matrices where the main randomness is Haar random matrices. Especially,  we consider that $X=U$ is either an $N \times N$ random Haar unitary or orthogonal matrix, so that {
\begin{equation}\label{eq_defndatamatrixtype}
\widehat Y=\widehat{A}^{1/2}U \widehat{B}^{1/2}.
\end{equation} }
We point out that the data matrix (\ref{eq_defndatamatrixtype}) has also appeared in the study of high dimensional data analysis, for instance, see \cite{7587390,bun2017,DW}.

\subsection{Some related results on finite rank deformation of random matrices}\label{sec_finiterankmatrixsummary}
In this section, we first pause to give a brief review of
the literature on the spectra of fixed-rank deformation of random matrices, a category of random matrix models including signal-plus-noise and spiked covariance matrices as typical examples. There exists rich literature in understanding the limiting behavior of the eigenvalues and eigenvectors of such deformed models. Since the seminal work of Baik, Ben Arous, and P{\' e}ch{\' e} \cite{BBP}, it is now well-understood that the extreme eigenvalues undergo the so-called BBP transition as the magnitude of the deformation changes. Roughly speaking, the extreme eigenvalues of the deformed matrix detaches from spectrum of the undeformed random matrix if and only if the strength of the deformation exceeds a certain threshold. In this case, we call the extreme eigenvalue as an outlier, and the associated eigenvector as an outlier eigenvector. In parallel to the outlier eigenvalues, an outlier eigenvector is concentrated on a cone with the axis parallel to the true eigenvectors (of the deformation) and the aperture explicitly determined by the deformations. Moreover, the remaining eigenvalues are close to those of the undeformed random matrices and the associated eigenvectors are delocalized.  

The results in the same spirit of the aforementioned arguments have been established for various deformed random matrix models under different settings when the random matrix part contains i.i.d. entries. On one hand, when the deformation is additive, the eigenvalues and eigenvectors have been studied for deformed Wigner matrices in \cite{2020arXiv200913143B, MR2782201, MR2489158, MR2919200, KY13,MR3262497}, for signal-plus-noise matrices in \cite{MR4206682, MR2944410,MR3298725, MR3837103, MR4036038} and for deformed non-Hermtian matrices in \cite{MR4283371,MR3500273, MR3552011, 10.1214/19-AIHP1002,  MR3010398}. On the other hand, when the deformation is multiplicative, the eigenvalues and eigenvectors have been investigated for spiked covariance matrices in \cite{MR2451053,MR2887686,BBP, BAIK20061382, BDWW,MR2782201, Bloemendal2016,MR3078286,DINGRMTA,MR2399865}, for spiked separable covariance matrices in \cite{DYaos}, for spiked CCA matrices in \cite{MR3909944,2021arXiv210203297M}, for spiked MANOVA matrices in \cite{MR4302572,MR3611497} and for spiked correlation matrices in \cite{MR4286186}.

When the randomness comes from Haar invariant random matrices, there are relatively fewer related works \cite{outliermodel,MR2782201,MR3792626}. All the existing works focus on finding the limits of the outlier eigenvalues and eigenvectors under stronger assumptions that the spikes are far away from the critical values by a distance of constant order. Consequently, they leave the convergent rates and the non-outlier eigenvalues and eigenvectors undiscussed.  The aim of this paper is to fill this gap by establishing the first order limits and precise rates of convergence for the outlier eigenvalues and eigenvectors  and {concentration} bounds for the non-outlier eigenvalues and eigenvectors for the model (\ref{eq_defndatamatrixtype}).





\subsection{An overview of our results}\label{sec_overviewofourresults}
In this subsection, we provide a rough overview of our results. Our theoretical findings are in the same spirit as those discussed in Section \ref{sec_finiterankmatrixsummary}. We pause to introduce some notations. It is well-known that for the unspiked model $A^{1/2} UBU^* A^{1/2},$ its empirical spectral distribution (ESD) is given by the \emph{free multiplicative convolution} of the ESDs $\mu_A$ and $\mu_B$ of $A$ and $B$ respectively, denoted as $\mu_A \boxtimes \mu_B$ \cite{Voiculescu1991}; see Definition \ref{def:freeconv} below for a precise definition. More recently, in our previous works \cite{DJ1,JHC}, we investigated the behavior of $\mu_A \boxtimes \mu_B$ by analyzing a pair of analytic functions, known as \emph{subordination functions}, $\Omega_A$ and $\Omega_B$ that define the free convolution; see (\ref{eq_defn_eq}) for details.  

We briefly describe our results, firstly on eigenvalues of $\wh{Y}\wh{Y}\adj$. Due to invariance, we only need to consider diagonal $\widehat{A}$ and $\widehat{B}$ whose entries are denoted in the decreasing order as $\{\widehat{a}_i\}$ and $\{\widehat{b}_j\}.$ We further assume that $\wh{A}$ and $\wh{B}$ contain spikes $\{\widehat{a}_i: 1 \leq i \leq r\}$ and $\{\widehat{b}_j: 1 \leq j \leq s\}$ with finite $r$ and $s$, respectively. A spike $\widehat{a}_i$ (or $\widehat{b}_j$) gives rise to an outlier of $\widehat{Y} \widehat{Y}^*$ if and only if $\widehat{a}_i>\Omega_{\beta}(E_{+})$ (or $\widehat{b}_j>\Omega_{\alpha}(E_{+})$), where $E_+$ is the rightmost edge of the support of $\mu_\alpha \boxtimes \mu_\beta.$ In this case, the outlier concentrates around a fixed location, namely $\Omega_B^{-1}(\widehat{a}_i)$ (or $\Omega_A^{-1}(\widehat{b}_j)$), where $\Omega_{A(B)}^{-1}(\cdot)$ are the inverse functions of the subordination functions; see Theorem \ref{thm_outlier} for more details. Our result also shows that this transition occurs on the scale $N^{-1/3}$, as in a typical BBP transition \cite{BBP}. More precisely, if $\widehat{a}_i - \Omega_{\beta}(E_{+}) \gg N^{-1/3}$ or $\widehat{b}_j - \Omega_{\alpha}(E_{+}) \gg N^{-1/3}$, that is, if the spike is supercritical, then the outlier will be well-separated from the support of $\mu_A \boxtimes \mu_B$ and can be detected readily. For $0< \widehat{a}_i - \Omega_{B}(E_+) \ll N^{-1/3}$ or $0< \widehat{b}_j - \Omega_{\alpha}(E_{+}) \ll N^{-1/3}$, that is, when the spike is subcritical, the corresponding "outlier" cannot be distinguished from the non-deformed spectrum and will instead stick to the right-most edge $E_{+}$ up to some random fluctuation of order $\rO(N^{-2/3})$. The rest of the non-outlier eigenvalues will stick to the eigenvalues of $A^{1/2} UBU^* A^{1/2}$; see Theorem \ref{thm_eigenvaluesticking} for more details. {We also remark that the convergence rates in Theorems \ref{thm_outlier} and \ref{thm_eigenvaluesticking} are consistent with \cite{BBP}, indicating that the second-order transition therein might also be true for our model. That is, under some mild assumptions, the outlier eigenvalues are asymptotically Gaussian whereas the extremal non-outlier eigenvalues follow Tracy-Widom distribution. We will pursue this direction in future works.}

Next for the sample eigenvector of $\widehat{Y}\widehat{Y}^*$ associated with the outlier caused by a supercritical spike $\widehat{a}_i$, we show that it is concentrated on a cone with axis parallel to the true eigenvector with an explicit aperture determined by $\widehat{a}_i$ and $\Omega_B$; see Theorem \ref{thm_outlieereigenvector} for more details.  On the other hand, the sample eigenvector of $\widehat{Y} \widehat{Y}^*$ that is associated with a subcritical spike $\widehat{b}_j$ is delocalized. Moreover, the non-outlier eigenvectors are also delocalized; see Theorem \ref{thm_nonoutliereigenvector} for more details. Similar results hold for the eigenvectors of $\widehat{Y}^* \widehat{Y}$ by switching the role of $\widehat{A}$ and $\widehat{B}.$ We remark that in \cite{outliermodel}, the authors have studied the convergent limits of the outlying eigenvalues and eigenvectors on a macroscopic scale; see  Remarks \ref{rem_eigenvaluesitkcing} and \ref{eq_strongeigenvectorassumption} for comparison. Our results extend the counterparts of \cite{outliermodel} to full strength, that is, with the optimal scale and rate.

Our proof adapts the same proof strategy as in \cite{Bloemendal2016, DYaos,KY13}. The key input is to establish the local laws both near and far away from the edge of the undeformed random matrices. {That is to say, we need precise concentration estimates for the resolvent of  $A^{1/2} UBU^* A^{1/2},$ denoted as $\widetilde{G}(z)=(A^{1/2} UBU^* A^{1/2}-z)^{-1}, \ z=E+\mathrm{i} \eta, \eta \gg N^{-1}.$ Especially, we need the entry-wise local law, i.e., concentration estimates for $\widetilde{G}_{ij}(z), 1 \leq i, j \leq N,$ and the averaged local law, i.e., concentration estimate for $N^{-1} \sum_{i=1}^N  \widetilde{G}_{ii}(z).$ In fact, all these concentration estimates depend on the subordination functions and the Stieltjes transform of $\mu_A \boxtimes \mu_B;$ see Theorem \ref{prop_linearlocallaw} for more details.} Once we have the local laws, we can reduce our tasks to studying the undeformed models using some perturbative approach; see Section \ref{sec_proofstartegy} for more details on the proof strategies.  We remark that we can also generalize the results for the invariant additive models of \cite{outliermodel} in a similar fashion by using and modifying the local laws for the additive model $A+UBU^*$ as in \cite{BEC}.



This paper is organized as follows. In Section \ref{sec:mainresult}, we formally introduce our model and some necessary assumptions. In Section \ref{sec_mainresultshere}, we provide the main results of the paper. In Section  \ref{sec_toolsandproofs}, we collect and prove some results which will be used for our proofs. We also provide a description of our proof strategy. The technical proofs are provided in Sections \ref{sec_proofofspikedmodel} and \ref{sec_proofeigenvectors}.

\vspace{3pt}
\noindent {\bf Conventions.}
	Throughout the rest of the paper, $N$ always denotes the size of our matrix model and we often omit the dependence on $N$. For $m,n\in\N$, we denote the set $\{k\in\N:m\leq k\leq n\}$ by $\llbra m,n\rrbra$. For $i\in\llbra 1,N\rrbra$, we denote by $\mathbf{e}_{i}$ the $(N\times 1)$ column vector with $(\mathbf{e}_{i})_{j}=\delta_{ij}$. We use $I$ for the identity matrix of any dimension without causing any confusion, and we abbreviate $\tr =N^{-1}\Tr$ for matrices with any dimension. For a matrix $A,$ we denote its operator norm by $\norm{A}.$ {Throughout the paper, we use $\mathbf{1}$ for the indicator function. } 
	
	Finally, we use the standard big-$\rO$ and little-$\ro$ notations; for $N$-dependent nonnegative numbers $a_{N}$ and $b_{N}$, we write $a_{N}=\rO(b_{N})$ or $a_N \lesssim b_N$ if $a_{N}/b_{N}$ remains bounded, $a_{N}\sim b_{N}$ if $a_{N}/b_{N}$ and $b_{N}/a_{N}$ are both bounded, and $a_{N}=\ro(b_{N})$ {(or equivalently $a_N \ll b_N$)} if $a_{N}/b_{N}$ converges to zero, with the convention $\cdot/0\deq 0$.  


\vspace{3pt}

\section{Definition of spiked multiplicative model}\label{sec:mainresult}
In this section, we introduce the model and some necessary assumptions. 
\subsection{Some notations and assumptions}

We first introduce some notations and technical assumptions. Throughout the paper, we will consistently use the following notations. For an $N \times N$ matrix $W$, its empirical spectral distribution (ESD) is denoted as  	
\begin{equation*}
\mu_W=\frac{1}{N}\sum_{i=1}^N \delta_{\lambda_i(W)}. 
\end{equation*}
The following transforms will be used frequently. 
\begin{defn}\label{defn_transform} For a probability measure $\mu$ defined on $\mathbb{R}_+,$ its \emph{Stieltjes transform} $m_{\mu}$ is defined as
\beqs
	m_{\mu}(z)\deq\int\frac{1}{x-z}\dd\mu(x), \ \ \text{for }z\in\CR.
\eeqs 
Moreover,  we define the $\mathtt{M}$-transform $M_{\mu}$ and $\mathtt{L}$-transform $L_{\mu}$ on $\CR$ as 
\begin{align}\label{eq_mtrasindenity}
	M_{\mu}(z)&\deq 1-\left(\int\frac{x}{x-z}\dd\mu(x)\right)^{-1}= \frac{zm_{\mu}(z)}{1+zm_{\mu}(z)}, \ \ \
	L_{\mu}(z)\deq \frac{M_{\mu}(z)}{z}.
\end{align}
\end{defn}
We first introduce the non-spiked model.  Consider two $N\times N$ real, deterministic, positive definite matrices 
	\beqs
	A\equiv A_{N}=\diag(a_{1},\cdots,a_{N}) \AND B\equiv B_{N}=\diag(b_{1},\cdots,b_{N}).
	\eeqs
	

We next introduce our assumptions on the undeformed model.

\begin{assu}\label{assu_limit} 
Throughout the paper, we assume that there exist two $N$-independent absolutely continuous probability measures $\mu_{\alpha}$ and $\mu_{\beta}$ on $(0, \infty)$ with densities respectively $\rho_{\alpha}$ and $\rho_{\beta}$ satisfying the following.
	\begin{itemize}
	\item[(i).] For simplicity, we assume that both of them have means $1$, that is, $\int x\dd\mu_{\alpha}(x)=\int x\rho_{\alpha}(x)\dd x=1.$
	\item[(ii).] Both $\rho_{\alpha}$ and $\rho_{\beta}$ have single non-empty intervals as supports, denoted as $[E_-^{\alpha}, E_+^{\alpha}]$ and $[E_-^{\beta}, E_+^{\beta}],$ respectively. Here $E_-^{\alpha}, E_+^{\beta}, E_-^{\alpha}$ and $E_+^{\beta}$ are all positive numbers. Moreover, both of $\rho_{\alpha}$ and $\rho_{\beta}$ are strictly positive  in the interior of their supports. 
	
	\item[(iii).]  There exist constants $-1<t^{\alpha}_{\pm},t^{\beta}_{\pm}<1$ and $C>1$ such that 
	\begin{align*}
		&C^{-1}\leq \frac{\rho_{\alpha}(x)}{(x-E_{-}^{\alpha})^{t_{-}^{\alpha}}(E_{+}^{\alpha}-x)^{t_{+}^{\alpha}}}\leq C,\quad\forall x\in[E_{-}^{\alpha},E_{+}^{\alpha}],\\
		&C^{-1}\leq \frac{\rho_{\beta}(x)}{(x-E_{-}^{\beta})^{t_{-}^{\beta}}(E_{+}^{\beta}-x)^{t_{+}^{\beta}}}\leq C,\quad\forall x\in[E_{-}^{\beta},E_{+}^{\beta}].
	\end{align*}
	\end{itemize}	
Moreover, for the ESDs of $A$ and $B,$ denoted as $\mu_A$ and $\mu_B,$ we assume the following:
\begin{itemize}
	\item[(iv).] For the Levy distance $\mathcal{L}(\cdot, \cdot),$  we have that for any small constant $\epsilon>0,$  when $N$ is sufficiently large
	\begin{equation}\label{eq_epsilondefinition}
	\bsd\deq \mathcal{L}(\mu_{\alpha}, \mu_A)+\mathcal{L}(\mu_{\beta}, \mu_{B}) \leq N^{-1+\epsilon},
	\end{equation}

	\item[(v).] For the supports 
of $\mu_{A}$ and $\mu_{B}$, we have that for any constant $\delta>0,$  when $N$ is sufficiently large
	\beqs
		\supp\mu_{A}\subset [E_{-}^{\alpha}-\delta,E_{+}^{\alpha}+\delta]\AND \supp\mu_{B}\subset[E_{-}^{\beta}-\delta,E_{+}^{\beta}+\delta].
	\eeqs
		
	
	\end{itemize}
\end{assu}




The first assumption is introduced for technical simplicity and it can be removed easily via scaling; see Remark 3.2 of \cite{JHC} for details. The single-cut assumption in (ii) guarantees that there are only two edges so that $m_{\mu_{H}}(z)$ is always monotone outside the spectrum; this is a technicality and one can still prove the same result around the uppermost edge in a multi-cut setting. Moreover, assumption (iii) is introduced to guarantee the square root behavior near the edges of the free multiplicative convolution of $\mu_{\alpha}$ and $\mu_{\beta}.$ When this condition fails, a different behavior of $\mu_{\alpha} \boxtimes \mu_{\beta}$ from our current discussion may arise; for example, see \cite{KLP,Lee2013ExtremalEA}. Assumption (iv) ensures that $\mu_{A}$ and $\mu_B$ converge respectively to $\mu_{\alpha}$ and $\mu_\beta$ sufficiently fast down to a local scale.  Finally, we may localize all of Assumption \ref{assu_limit} to the upper edge and weaken (\ref{eq_epsilondefinition}) as far as eigenvalues and eigenvectors around the upper edge are concerned; see Remark \ref{rem_localize} for details. \normalcolor

The central results of Voiculescu in \cite{Voiculescu1987,Voiculescu1991} implies that $\mu_{H}$ converges weakly to a deterministic measure under Assumption \ref{assu_limit}, denoted as $\mu_{\alpha}\boxtimes\mu_{\beta}$. It is called the \emph{free multiplicative convolution} of $\mu_{\alpha}$ and $\mu_{\beta}$.
In the present paper, we use the $\mathsf{M}$-transform in (\ref{eq_mtrasindenity}) and associated subordination functions to define the free multiplicative convolution.\normalcolor
%
%
\begin{lem}[Proposition 2.5 of \cite{JHC}]\label{lem_subor}
	For Borel probability measures $\mu_{\alpha}$ and $\mu_{\beta}$ on $\R_{+}$, there exist unique analytic functions $\Omega_{\alpha},\Omega_{\beta}:\CR\to\CR$ satisfying the following: \\
\noindent{(1).} {For all $z\in\C_{+}$, we have $\Omega_{\alpha}(z),\Omega_{\beta}(z)\in\C_{+}$, $\Omega_{\alpha}(\ol{z})=\ol{\Omega_{\alpha}(z)}$, $\Omega_{\beta}(\ol{z})=\ol{\Omega_{\beta}(z)}$,}
		\beq\label{eq_subsys3}
		\arg \Omega_{\alpha}(z)\geq \arg z, \AND \arg\Omega_{\beta}(z)\geq \arg z.
		\eeq
\noindent{(2).}We have
		\beq \label{eq_subsys2}
		\lim_{x\searrow-\infty}\Omega_{\alpha}(x)=\lim_{x\searrow-\infty}\Omega_{\beta}(x)=-\infty.
		\eeq
		
\noindent{(3).} For all $z\in\CR$, we have 
		\beq\label{eq_suborsystem}
		zM_{\mu_{\alpha}}(\Omega_{\beta}(z))=zM_{\mu_{\beta}}(\Omega_{\alpha}(z))=\Omega_{\alpha}(z)\Omega_{\beta}(z).
		\eeq
\end{lem}

{We remark that Lemma \ref{lem_subor} is originally due to \cite{Belinschi-Bercovici2007}, with a slightly different notation using $\eta$-transform defined by $\eta(z)=1/M(1/z)$. Later in \cite{Chistyakov-Gotze2011}, exactly the same result as Lemma \ref{lem_subor} was proved with different notations. The aforementioned papers also considered analytic subordination functions for free additive and multiplicative convolutions on $\R$ and the circle, respectively. }

\begin{defn}\label{def:freeconv}
Denote the analytic function $M: \mathbb{C} \backslash \mathbb{R}_+ \rightarrow \mathbb{C} \backslash \mathbb{R}_+ $ by
\beq\label{eq_defn_eq}
M(z):=M_{\mu_{\alpha}}(\Omega_{\beta}(z))=M_{\mu_{\beta}}(\Omega_{\alpha}(z)).
\eeq
	The {free multiplicative convolution} of $\mu_{\alpha}$ and $\mu_{\beta}$ is defined as the unique probability measure $\mu,$ 
	denoted as $\mu \equiv \mu_{\alpha}\boxtimes\mu_{\beta}$ such that (\ref{eq_defn_eq}) holds
	 for all $z\in\CR$. In this sense, $M(z) \equiv M_{\mu_{\alpha} \boxtimes \mu_{\beta}}(z)$ is the $\mathsf{M}$-transform of $\mu_{\alpha} \boxtimes \mu_{\beta}.$ Furthermore, the analytic functions $\Omega_{\alpha}$ and $\Omega_{\beta}$  are referred to as the  subordination functions. 
Similarly we define $\Omega_{A}$ and $\Omega_{B}$ by replacing $(\alpha,\beta)$ with $(A,B)$ in Lemma \ref{lem_subor}, and define $\mu_{A}\boxtimes\mu_{B}$ so that $M_{\mu_{A}}(\Omega_{B}(z))=M_{\mu_{B}}(\Omega_{A}(z))=M_{\mu_{A}\boxtimes\mu_{B}}(z)$ for all $z\in\CR$.
\end{defn}
{
We conclude this subsection with preliminary facts on the free convolutions $\mu_{\alpha}\boxtimes\mu_{\beta},\mu_{A}\boxtimes\mu_{B}$ and their associated subordination functions, that will be required in the precise statements of our main results. In actual proofs, we also use more detailed description from Lemma \ref{prop:stabN} below. 
\begin{lem}\label{lem_prop_fc}
	Suppose that $\mu_{\alpha},\mu_{\beta},\mu_{A},$ and $\mu_{B}$ satisfy Assumption \ref{assu_limit}. 
	\begin{itemize}
		\item[(i).] The restrictions $\Omega_{\alpha}\vert_{\C_{+}},\Omega_{\beta}\vert_{\C_{+}},\Omega_{A}\vert_{\C_{+}},\Omega_{B}\vert_{\C_{+}}$ extend continuously to $[0,\infty)$, mapping into the Riemann sphere $\C\cup\{\infty\}$; we consistently use the same notations $\Omega_{\alpha},\Omega_{\beta}$, etc., to denote these extensions. Furthermore, $\Omega_{\alpha}$ and $\Omega_{\beta}$ are bounded on each compact subset of $\C_{+}\cup\R$.
		\item[(ii).] The free convolution $\mu_{\alpha}\boxtimes\mu_{\beta}$ has a continuous and bounded density $\rho$. The density $\rho$ is supported on a compact interval $[E_{-},E_{+}]$ in $(0,\infty)$ and satisfies
		\beq\label{eq_edge}
			\rho(x) \sim \sqrt{(E_{+}-x)(x-E_{-})},\qquad x\in[E_{-},E_{+}].
		\eeq
Here $E_+:=\sup \operatorname{supp} \mu_\alpha \boxtimes \mu_\beta.$		
		\item[(iii).] There exist positive constants $C_{\alpha}$ and $C_{\beta}$ such that, for each fixed neighborhood $U$ of $E_{+}$, 
		\begin{equation}\label{eq_localquadtratic}
			\Omega_{\alpha(\beta)}(z)=\Omega_{\alpha(\beta)}(E_+)+C_{\alpha(\beta)} \sqrt{z-E_+}+\rO(|z-E_+|^{3/2}), \qquad z\in U.
		\end{equation}
		
		\item[(iv).] We have
		\beqs
		\dist(\Omega_{\alpha}(\CR),\supp\mu_{\beta})>0,\qquad \dist(\Omega_{\beta}(\CR),\supp\mu_{\alpha})>0.
		\eeqs
		
		\item[(v).] The maps $\Omega_{\alpha},\Omega_{\beta}$ are real analytic and strictly increasing on $(E_{+},\infty)$, and satisfy 
		\beqs
			\lim_{x\to\infty,x\in\R}\Omega_{\alpha}(x)=\infty=\lim_{x\to\infty,x\in\R}\Omega_{\beta}(x).
		\eeqs The same results hold true for $\Omega_{A},\Omega_{B}$ if we replace $E_{+}$ by
		\beqs
			\wh{E}_{+}\deq \sup\supp\mu_{A}\boxtimes\mu_{B}.
		\eeqs
		
		\item[(vi).] For any fixed $\epsilon>0$, we have
			\beqs
				E_{+}-N^{-2/3+\epsilon}\leq\wh{E}_{+}\leq E_{+}+N^{-1+\epsilon}.
			\eeqs
	\end{itemize}
\end{lem}
\begin{proof}
	The first statement is due to Proposition 4.3 and Lemma 5.2 of \cite{JHC}. The second to fourth statements are proved in Theorem 3.3 and Propositions 5.10 and 5.6 of \cite{JHC}, in that order. The fifth and sixth statements are immediate consequences of Lemmas A.8 and A.7 of \cite{DJ1}.
\end{proof}
}

\subsection{The model}

With the above preparation, we introduce the spiked model following the setup  of \cite{DINGRMTA,DYaos}. Throughout the paper, we consider the spikes associated with the upper edge $E_+.$ Similar discussion applies to the lower edge $E_-.$ To add a few spikes,  we assume that there exist some fixed integers $r$ and $s$ with two sequences of  positive numbers $\{d_i^a\}_{i \leq r}$ and $\{d_j^b\}_{j \leq s}$ such that $\wh{A}=\operatorname{diag}\{\wh{a}_1, \cdots, \wh{a}_N\}$ and $\wh{B}=\operatorname{diag}\{\wh{b}_1, \cdots, \wh{b}_N\},$ where 
\begin{equation}\label{eq_spikes}
\wh{a}_k=
\begin{cases}
a_k(1+d^a_k), & 1 \leq k \leq r \\
a_k, & k \geq r+1
\end{cases}, \  
\wh{b}_k=
\begin{cases}
b_k(1+d^b_k), & 1 \leq k \leq s \\
b_k, & k \geq s+1
\end{cases}.
\end{equation}
Without loss of generality, we assume that $\wh{a}_1 \geq \wh{a}_2 \cdots \geq \wh{a}_N$ and $\wh{b}_1 \geq \wh{b}_2 \cdots \geq \wh{b}_N.$ In the current paper, we assume that all the $d_k^a$'s and $d_k^b$'s are bounded. 


Let $\Omega_A(\cdot)$ and $\Omega_B(\cdot)$ be the subordination functions associated with $\mu_A$ and $\mu_B$. We will see that 
a spike $\wh{a}_i, 1 \leq i \leq r$ or $\wh{b}_j, 1 \leq j \leq s,$ causes an outlier eigenvalue, if {
\begin{equation}\label{eq_outlierlocation}
\wh{a}_i>\Omega_{\beta}(E_{+}), \ \text{or} \ \ \wh{b}_j>\Omega_{\alpha}(E_{+}). 
\end{equation} }
More precisely, we will use the following assumption. 
\begin{assu}\label{assum_outlier} We assume that (\ref{eq_outlierlocation}) holds for all $1 \leq i \leq r$ and $1 \leq j \leq s.$ Moreover, we define the integers $0 \leq r^+\leq r$ and $0 \leq s^+\leq s$ by {
\begin{align}\label{eq_rsplusedefn}
	r^{+}\deq \max\{ 1 \leq i \leq r:\wh{a}_{i}\geq \Omega_{\beta}(E_{+})+N^{-1/3}\},  \ \
	s^{+}\deq \max\{1 \leq j \leq s:\wh{b}_{j}\geq \Omega_{\alpha}(E_{+})+N^{-1/3}\},
\end{align} }
and take their value to be zero when the index sets are empty. The lower bound $N^{-1/3}$ is chosen for definiteness, and it can be replaced
with any $N$-dependent parameter that is of the same order.  
\end{assu}

\begin{rem}\label{rem_r+}
A spike $\wh{a}_i$ or $\wh{b}_j$ that does not satisfy the conditions in Assumption \ref{assum_outlier} will cause an outlier eigenvalue that lies within an $\rO(N^{-2/3})$ neighborhood of the edge $E_+.$ In this sense, it will be hard to detect such a spike as $\rO(N^{-2/3})$ exactly matches the scale of eigenvalue spacings around the edge in the non-spiked model as demonstrated in Lemma \ref{thm_rigidity}. In this sense, Assumption \ref{assum_outlier} simply chooses the actual spikes. In the statistical literature, this is referred to as the supercritical regime and a reliable detection of the spikes is only available in this regime. We refer the readers to \cite{BDWW,Bloemendal2016, DYaos,perry2018} for more detailed discussion.    
\end{rem}


Throughout the paper, for the ease of discussion, we will consistently use the following notations
\begin{equation}\label{matrix_defn_defn_null}
\mathcal{Q}_1 \equiv \widetilde{H}:=A^{1/2}UBU^* A^{1/2},
\end{equation}
and for $\widehat{Y}$ is defined in (\ref{eq_defndatamatrixtype}) {
\begin{equation}\label{matrixnotation}
\widehat{\mathcal{Q}}_1:=\widehat{Y} \widehat{Y}^* \equiv \widehat{A}^{1/2} U \widehat{B} U^* \widehat{A}^{1/2},  \ \ \widehat{\mathcal{Q}}_2:=\widehat{Y}^* \widehat{Y} \equiv \widehat{B}^{1/2} U^* \widehat{A} U \widehat{B}^{1/2}. 
\end{equation} }
\begin{rem}
We provide a remark on how to estimate the quantities in (\ref{eq_outlierlocation}). We use $E_+$ and $\Omega_\beta(E_+)$ for examples.  First, according to (\ref{eq_nnnnnnlilililililili}) and (\ref{eq_stickingbound}) below, $E_+$ can be efficiently estimated using the largest eigenvalue of $\mathcal{Q}_1,$ denoted as $\lambda_1(\mathcal{Q}_1),$ or the first non-outlier eigenvalue of $\widehat{\mathcal{Q}}_1.$ Second, according Lemmas A.3 and A.7 and (C.34) of \cite{DJ1}, we find that $|\Omega_{B}^{c}(\wh{z})-\Omega_{\beta}(E_{+})|\prec N^{-1/3},$ where for some sufficiently small constant $\epsilon>0,$ we denote 
	\beqs
	\Omega_{B}^{c}(\wh{z})\deq \frac{\wh{z}\tr A\wt{G}(\wh{z})}{1+\wh{z}\tr \wt{G}(\wh{z})},\qquad \text{where} \ \ \wh{z}=\lambda_{1}(\caQ_{1})+\ii N^{-2/3+\epsilon} \ \ \text{and} \ \ \wt{G}(z)\deq (\wt{H}-z)^{-1}.
	\eeqs
This shows that $\operatorname{Re} \Omega_{B}^{c}(\wh{z})$ can serve as a consistent estimator for $\Omega_{\beta}(E_{+}).$ 
\end{rem}\normalcolor

\section{Main results}\label{sec_mainresultshere}
In this section, we state the main results. Throughout the paper, we will consistently use the notion of \emph{stochastic domination}, which was firstly introduced in \cite{MR3119922}. It simplifies the presentation by systematizing the statements of the form "$X_N$ is bounded by $Y_N$ with high probability up to a small power of $N$."   
\begin{defn} For two sequences of random variables $\{X_N\}_{N \in \mathbb{N}}$ and $\{Y_N\}_{N \in \mathbb{N}},$ we say that $X_N$ is stochastically dominated by $Y_N,$ written as $X_N  \prec Y_N$ or $X_N =\rO_{\prec}(Y_N),$ if for all (small) $\epsilon>0$ and (large) $D>0,$ we have 
\begin{equation*}
\mathbb{P}\left( |X_N| \geq N^{\epsilon} |Y_N| \right) \leq N^{-D},
\end{equation*}
for sufficiently large $N \geq N_0(\epsilon, D).$ If $X_N(\upsilon)$ and $Y_N(\upsilon)$ depend on a common parameter $\upsilon,$ we say $X_N \prec Y_N$ uniformly in $\upsilon$ if the threshold $N_0(\epsilon, D)$ can be chosen independently of the parameter $\upsilon.$ Moreover, we say an event $\Xi$ holds with high probability if for any constant $D>0,$ $\mathbb{P}(\Xi) \geq 1-N^{-D}$ for large enough $N.$ 
\end{defn}

\subsection{Eigenvalue statistics}\label{sec_spikedmodel}
In this subsection, we state our results regarding the eigenvalue statistics. The statements often involve the inverse functions $\Omega_{A}^{-1}$ and $\Omega_{B}^{-1}$; since $\Omega_{A}$ and $\Omega_{B}$ are monotone increasing and real analytic on $(\wh{E}_{+},\infty)$ by Lemma \ref{lem_prop_fc}, we denote by $\Omega_A^{-1}(\cdot)$ and $\Omega_B^{-1}(\cdot)$ as the inverse functions of $\Omega_A$ and $\Omega_B$ on complex neighborhoods of $(\Omega_{A}(\wh{E}_{+}), \infty)$ and $(\Omega_{B}(\wh{E}_{+}), \infty),$ respectively. When $x<\Omega_{A}(\wh{E}_{+})$ and $y<\Omega_{B}(\wh{E}_{+})$, we use the conventions $\Omega_{A}^{-1}(x)=E_{+}$ and $\Omega_{B}^{-1}(y)=E_{+}$. \normalcolor  For the ease of statements, we introduce the following re-labeling for the eigenvalues of $\wh{\mathcal{Q}}_{i}$'s as in \cite[Definition 3.5]{DYaos}. \normalcolor

\begin{defn} We define the labeling functions $\pi_a,\pi_{b}:\llbra1,N\rrbra\to\llbra1,N\rrbra$ as follows. For any $1\leq i \leq r$, we assign to it a label $\pi_a(i)\in \{1,\cdots, r+s\}$ if $\Omega_B^{-1}(\wh{a}_i)$ is the $\pi_a(i)$-th largest element in $\{\Omega_B^{-1}(\wh{a}_i)\}_{i=1}^r \cup  \{\Omega^{-1}_{A}(\wh{b}_j)\}_{j=1}^s$. We also assign to any $1\leq j \leq s$ a label $\pi_b(j)\in \{1,\cdots, r+s\}$ in a similar way. Moreover, we define $\pi_a(i)=\min(i+s,N)$ if $i>r$ and similarly for $\pi_{b}(j)$.
We define the following sets of outlier indices:
\begin{align*}
& \mathcal O:= \{\pi_a(i): 1\le i \le r\}\cup \{\pi_b(j): 1\le j \le s\}, 
\end{align*}
and
\begin{align*}
& \mathcal O^+:= \{\pi_a(i): 1\le i \le r^+\}\cup \{\pi_b(j): 1\le j \le s^+\},
\end{align*}
{where we recall the definitions of $r^+$ and $s^+$ in (\ref{eq_rsplusedefn}). }
\end{defn}

First, we state the results on the convergence limits and rates for the outliers and the first few non-outlier eigenvalues. Recall (\ref{matrixnotation}). Denote the eigenvalues of $\wh{\mathcal{Q}}_1$ and $\wh{\mathcal{Q}}_2$ as $\hlam_1 \geq \hlam_2 \geq \cdots \geq \hlam_N.$ 
\begin{thm}[Outlier and extremal non-outlier eigenvalues]\label{thm_outlier} Suppose Assumptions \ref{assu_limit} and \ref{assum_outlier} hold. Then we have that 
\begin{equation*}
\left| \hlam_{\pi_a(i)}-\Omega_B^{-1}(\wh{a}_i) \right| \prec N^{-1/2}(\wh{a}_i-\Omega_{\beta}(E_{+}))^{1/2}, \ 1 \leq i \leq r^+,
\end{equation*}
and 
\begin{equation*}
\left| \hlam_{\pi_b(j)}-\Omega^{-1}_A(\wh{b}_j) \right| \prec N^{-1/2}(\wh{b}_j-\Omega_{\alpha}(E_{+}))^{1/2}, \ 1 \leq j \leq s^+.
\end{equation*}
Moreover, for any fixed integer $\varpi>r+s,$ we have 
\begin{equation}\label{eq_nnnnnnlilililililili}
\left|\hlam_i-E_+ \right| \prec N^{-2/3}, \ \text{for} \ i \notin \mathcal{O}^+ \ \text{and} \ i \leq \varpi.  
\end{equation} 
\end{thm}

Theorem \ref{thm_outlier} offers the {concentration} bounds for the locations of the outlier and first few extremal non-outlier eigenvalues. It shows that the convergence rates of the outliers change from the order of $N^{-1/2}(\wh{a}_i-\Omega_{\beta}(E_{+}))^{1/2}$ or  $N^{-1/2}(\wh{b}_j-\Omega_{\alpha}(E_{+}))^{1/2}$ to $N^{-2/3}$ once $\wh{a}_i-\Omega_{\beta}(E_{+})$ or  $\wh{b}_j-\Omega_{\alpha}(E_{+})$ cross the scale $N^{-1/3},$ as opposed to Assumption \ref{assum_outlier}. This indicates the occurrence of  BBP transition \cite{BBP}.  Moreover, we believe that under some spectral gap assumptions as in \cite{MR2451053,MR2887686,BBP,BDWW,Bloemendal2016,9779233}, the  distributions of the outlier eigenvalues can also be studied.  We will pursue this direction in future works. 

\begin{rem}\label{rem_eigenvaluesitkcing}
We remark that the locations, in other words, convergence limits of the outlier eigenvalues have been obtained without Assumption \ref{assu_limit} in \cite[Section 2.2]{outliermodel} for the spiked unitarily invariant model. However, in \cite{outliermodel}, the conditions in Assumption \ref{assum_outlier} are stronger in the sense that
\begin{equation}\label{eq_strongone}
\wh{a}_i \geq \Omega_{\beta}(E_{+})+\varsigma \ \text{if and only if} \ 1 \leq i \leq r,
\end{equation}
and 
\begin{equation}\label{eq_strongtwo}
\wh{b}_j \geq \Omega_{\alpha}(E_{+})+\varsigma \ \text{if and only if} \ 1 \leq j \leq s,
\end{equation} 
where $\varsigma>0$ is some fixed constant. Compared to \cite{outliermodel}, we extend the results on the outlier eigenvalues in the following aspects: We consider spikes on a finer scale in Assumption \ref{assum_outlier}, establish their convergent rates, and provide the results for the extremal non-outlier eigenvalues. We believe that Assumption \ref{assum_outlier} is the most general assumption possible for the existence of the outliers, and the convergent rates obtained here are optimal up to an $N^{\epsilon}$ factor, where $\epsilon>0$ can be any (small) constant. 

\end{rem}

{
\begin{rem}
	On macroscopic level, \cite{outliermodel} covers more general settings assuming (\ref{eq_strongone}) and (\ref{eq_strongtwo}). Firstly, the limiting measures $\mu_{\alpha}$ and $\mu_{\beta}$ are only assumed to be non-trivial and compactly supported. In this general setting the resulting set of outliers is (see \cite[Theorem 2.2]{outliermodel})
	\beq
		\left(\bigcup_{1\leq i\leq r} \Omega_{B}^{-1}(\{\wh{a}_{i}\})
		\cup\bigcup_{1\leq j\leq s} \Omega_{A}^{-1}(\{\wh{b}_{j}\})\right)
		\setminus \supp (\mu_{\alpha}\boxtimes\mu_{\beta}),
	\eeq
	which is consistent with our Theorem \ref{thm_outlier}. It should be noted that, under Assumption \ref{assu_limit}, each set $\Omega_{B}^{-1}(\{a_{i}\})$ necessarily has at most one element. In other words, each $\wh{a}_{i}$ or $\wh{b}_{j}$ gives rise to at most one outlier. To see this, recall the following three facts from Lemma \ref{lem_prop_fc}; (i) $\Omega_{\beta}$ is an increasing real analytic function on $\R\setminus\supp(\mu_{\alpha}\boxtimes\mu_{\beta})$ with $\Omega_{\beta}(\pm\infty)=\pm\infty$, (ii) $\im\Omega_{\beta}>0$ in the interior of $\supp(\mu_{\alpha}\boxtimes\mu_{\beta})$, and (iii) $\Omega_{\beta}(E_{+})> E_{+}^{\beta}$ and $\Omega_{\beta}(E_{-})<E_{-}^{\beta}$. From these facts and that $\mu_{\alpha}\boxtimes\mu_{\beta}$ is supported on a single interval, we easily find that each of $\Omega_{\beta}^{-1}(\wh{a}_{i})$ has at most one element. 
	
	However, for general $\mu_{\alpha}$ and $\mu_{\beta}$, a single spike $\wh{a}_{i}$ or $\wh{b}_{j}$ may result in several outliers; see \cite[Examples 2.3 and 2.4]{outliermodel} for concrete examples. We believe that the three facts above extend to fairly general $\mu_{\alpha}$ and $\mu_{\beta}$, but their free convolution may have disconnected support, in particular if either $\mu_{\alpha}$ or $\mu_{\beta}$ does. In this case, each interval between connected components of $\supp(\mu_{\alpha}\boxtimes\mu_{\beta})$ may contain a solution $x$ of the equation $\Omega_{B}(x)=\wh{a}_{i}$, in contrast to our case where the support is connected. 
	
Finally, we point out that \cite{outliermodel} also studied the additive and circular models, namely $A+UBU\adj$ and $AUBU\adj$ for Hermitian and unitary matrices $A,B$, respectively. We believe the counterpart of our results for the additive model can be proved analogously, since the two major inputs, the square-root decay of $\mu_{\alpha} \boxplus \mu_\beta$ near the edges and the optimal edge local laws, have already been proved in \cite{BEC}. As for the circular model, neither of the two inputs is known to our knowledge; the latest result on the free multiplicative convolution on the circle is \cite{Belinschi-Bercovici-Ho2022}, which shows absolute continuity. We will pursue this direction in the future works. 
\end{rem}
}

Second, we study the non-outlier eigenvalues of $\widehat{\mathcal{Q}}_1$ for $i>r^++s^+.$ These eigenvalues are governed by the \emph{eigenvalue sticking} property, which states that the non-outlier eigenvalues of $\widehat{\mathcal{Q}}_1$ "stick"  with those of $\mathcal{Q}_1$ defined in (\ref{matrixnotation}) with high probability. Recall that we denote the eigenvalues of $\mathcal{Q}_1$ as $\lambda_1 \geq \lambda_2 \geq \cdots \geq \lambda_N.$

\begin{thm}[Eigenvalue sticking]\label{thm_eigenvaluesticking} Suppose Assumptions \ref{assu_limit} and \ref{assum_outlier} hold. Define 
\begin{equation}\label{eq_defngamma}
\gamma:=\min\left\{ \min_i|\widehat{a}_i-\Omega_{\beta}(E_{+})|, \ \min_j|\widehat{b}_j-\Omega_{\alpha}(E_{+})| \right\}.
\end{equation}
Fix any small constant $\tau>0.$ We have that
\begin{equation}\label{eq_stickingbound}
\left| \widehat{\lambda}_{i+r^++s^+}-\lambda_i \right| \prec \frac{1}{N \gamma}, \ 1 \leq i \leq \tau N. 
\end{equation} 
\end{thm}

Theorem \ref{thm_eigenvaluesticking} provides the convergence limits and rates for the non-outlier eigenvalues of $\widehat{\mathcal Q}_1$ in terms of those of $\mathcal Q_1.$ On one hand, if $\gamma \sim 1, $ together with the rigidity of $\mathcal{Q}_1$ established in \cite{DJ1} (or see Lemma \ref{thm_rigidity}), we conclude that the non-outlier eigenvalues of $\widehat{\mathcal{Q}}_1$ will converge to the quantiles of $\mu_{A} \boxtimes \mu_{B}$ with optimal rates. On the other hand, once $\gamma \gg N^{-1/3}$ (which is slightly stronger than the condition in Assumption \ref{assum_outlier}), for $i=\rO(1),$ we find that the right-hand side of (\ref{eq_stickingbound}) is much smaller than $N^{-2/3}$ obtained for the edge eigenvalues of $\mathcal{Q}_1$ as in Lemma \ref{thm_rigidity}. In future works, we will show that the edge eigenvalues of $\mathcal{Q}_1$ follow Tracy-Widom distribution, which immediately implies the Tracy-Widom asymptotics of the largest non-outlier eigenvalue of $\widehat{\mathcal{Q}}_1.$


\subsection{Eigenvector statistics}\label{sec_eigenvectorresults}
In this subsection, we introduce the results regarding the singular vectors, that is, the eigenvectors of $\widehat{\mathcal{Q}}_1$ and $\widehat{\mathcal{Q}}_2$ in (\ref{matrixnotation}). In what follows, we consider an index set $S$ such that  
\begin{equation}\label{eq_setscondition}
S \subset \mathcal{O}^+.
\end{equation}
For convenience, we use the following notations. For $1\le i_1 \le r^+$, $1\le  i_2 \le N$ and $1\le j \le N$, we define
 \begin{equation}\label{eq_differencedefinition}
\delta_{\pi_a(i_1), \pi_a(i_2)}^{a}:=|\wh{a}_{i_1}-\wh{a}_{i_2}|, 
\quad \delta_{\pi_a(i_1), \pi_b(j)}^{a}:=
\left| \wh{b}_j - \Omega_{A}(\Omega_B^{-1}( \wh{a}_{i_1}))\right|.
\end{equation}
Similarly, for $1\le j_1 \le s^+$, $1\le  j_2 \le N$ and $1\le i \le N$, we define
 \begin{equation*}
\delta_{\pi_b(j_1), \pi_a(i)}^{b}:=|\wh{a}_i - \Omega_B(\Omega_A^{-1}(\wh{b}_{j_1}))|, \quad \delta_{\pi_b(j_1), \pi_b(j_2)}^{b}:=|\wh{b}_{j_1}-\wh{b}_{j_2}|.
\end{equation*}
Further if $\mathfrak a\in S$, we define
\begin{equation}\label{eq_alphasinside}
\delta_{\mathfrak a}(S):=\begin{cases}\left( \min_{ k:\pi_a(k)\notin S}\delta^a_{\mathfrak a, \pi_a(k)}\right)\wedge \left( \min_{j:\pi_b(j)\notin S}\delta^a_{\mathfrak a, \pi_b(j)}\right), & \ \text{if } \mathfrak a=\pi_a(i) \in S\\
\left( \min_{k:\pi_a(k)\notin S}\delta^b_{\mathfrak a, \pi_a(k)}\right)\wedge \left( \min_{j:\pi_b(j)\notin S}\delta^b_{\mathfrak a,\pi_b(j)}\right), & \ \text{if } \mathfrak a=\pi_b(j) \in S
\end{cases};
\end{equation}
if $\mathfrak a\notin S$, then we define
\begin{equation}\label{eq_alphaoutside}
\delta_{\mathfrak a}(S):=\left( \min_{k:\pi_a(k)\in S}\delta^a_{\pi_a(k), \mathfrak a}\right)\wedge \left( \min_{j:\pi_b(j)\in S}\delta^b_{\pi_b(j), \mathfrak a}\right).
\end{equation}

With the above preparation, we proceed to state our main results on the outlier singular vectors. Denote the projections
\begin{equation*}
\mathcal{P}_S= \sum_{k \in S} \wh{\ub}_k \wh{\ub}_k^*, \ \text{and} \ \mathcal{P}'_{S}= \sum_{k \in S} \wh{\vb}_k \wh{\vb}_k^*,
\end{equation*}
where $\{\wh{\ub}_k\}$ and $\{\wh{\vb}_k\}$ are the eigenvectors of $\widehat{\mathcal{Q}}_1$ and $\widehat{\mathcal{Q}}_2$ in (\ref{matrixnotation}), respectively.

\begin{thm}[Outlier eigenvectors]\label{thm_outlieereigenvector} Suppose that Assumptions \ref{assu_limit} and \ref{assum_outlier} hold. For the set $S$ in (\ref{eq_setscondition}) and any given deterministic vector $\mathbf{v}=(v_1, \cdots, v_N)^* \in \mathbb{C}^N,$ we have that for the left singular vectors, 
\begin{equation*}
\left| \langle \mathbf{v},  \mathcal{P}_S \mathbf{v}\rangle-g_a(\mathbf{v}, S) \right| \prec \sum_{i: \pi_a(i) \in S} \frac{|v_i|^2}{\sqrt{N(\wh{a}_i-\Omega_{\beta}(E_{+}))}}+\sum_{i=1}^N \frac{|v_i|^2}{N \delta_{\pi_a(i)}(S)}+g_a(\mathbf{v},S)^{1/2} \left( \sum_{\pi_a(i) \notin S} \frac{|v_i|^2}{N \delta_{\pi_a(i)}(S)} \right)^{1/2},
\end{equation*}
where $g_a(\mathbf{v}, S)$ is defined as 
\begin{equation}\label{eq_concentrationlimit}
g_a(\mathbf{v}, S):=\sum_{i: \pi_a(i) \in S} \wh{a}_i \frac{(\Omega_B^{-1})'(\wh{a}_i)}{\Omega_B^{-1}(\wh{a}_i)}|v_i|^2. 
\end{equation}
 Similarly, for the right singular vectors, we have 
 \begin{equation*}
 \left| \langle \mathbf{v}, \mathcal{P}_S' \mathbf{v}\rangle -g_b(\mathbf{v}, S) \right| \prec \sum_{i: \pi_b(j) \in S} \frac{|v_i|^2}{\sqrt{N(\wh{b}_j-\Omega_{\alpha}(E_{+}))}}+\sum_{j=1}^N \frac{|v_j|^2}{N \delta_{\pi_b(j)}(S)}+g_b(\mathbf{v},S)^{1/2} \left( \sum_{\pi_b(j) \notin S} \frac{|v_j|^2}{N \delta_{\pi_b(j)}(S)} \right)^{1/2} ,
 \end{equation*}
 where $g_b(\mathbf{v},S)$ is defined as 
 \begin{equation*}
 g_b(\mathbf{v}, S):=\sum_{j: \pi_b(j) \in S} \wh{b}_j \frac{(\Omega_A^{-1})'(\wh{b}_j)}{\Omega_A^{-1}(\wh{b}_j)}|v_j|^2. 
 \end{equation*}
\end{thm}

Theorem \ref{thm_outlieereigenvector} establishes the {concentration} bounds for the generalized components of the outlier singular vectors. It demonstrates that the singular vectors are concentrated on a cone with axis parallel to the true singular vectors with an explicit aperture depending on the spikes and subordination functions.  We consider an example for illustration. For simplicity, we consider the non-degenerate case such that all the outliers are well-separated in the sense that we can simply choose $S=\{\pi_a(i)\}$ or $S=\{\pi_b(j)\}.$ Let $S=\{\pi_a(i)\}$ and $\vb=\mathbf{e}_i.$ Then we obtain from Theorem \ref{thm_outlieereigenvector} that 
\begin{equation}\label{eq_exampleeqone}
\left| \langle \wh{\mathbf{u}}_i, \mathbf{e}_i \rangle \right|^2=\wh{a}_i \frac{(\Omega_B^{-1})'(\wh{a}_i)}{\Omega_B^{-1}(\wh{a}_i)}+\rO_{\prec}\left(\frac{1}{\sqrt{N}(\wh{a}_i-\Omega_{\beta}(E_{+}))^{1/2}}+\frac{1}{N \delta^2_i} \right), \ \delta_i=\delta_{\pi_a(i)}(\pi_a(i)).
\end{equation}
It is easy to see from (\ref{conidition_nonovertwo}),  (\ref{eq_complexedgederterministicbound}) and (\ref{eq_derivativeinversebound}) that the error term is much smaller than the first term of the right-hand side of (\ref{eq_exampleeqone}). In this sense, $\wh{\ub}_i$ is concentrated on a cone with axis parallel to $\mathbf{e}_i.$   

\begin{rem}\label{eq_strongeigenvectorassumption}
We mention that some partial results of Theorem \ref{thm_outlieereigenvector} have been obtained in 
(4) of Theorem 2.5 in \cite{outliermodel} for the spiked unitarily invariant model. More specifically, by assuming that $r=0$ or $s=0,$ (\ref{eq_strongone}) and (\ref{eq_strongtwo}), they obtained the concentration limit (\ref{eq_concentrationlimit}).  We extend the counterparts in \cite{outliermodel} by, on one hand, stating the results in a less restrictive setting, and on the other hand, establishing their convergent rates.

%
\end{rem}

Finally, we state the main results regarding the non-outlier singular vectors. Denote
\begin{equation*}
\kappa_i:=i^{2/3} N^{-2/3}.
\end{equation*}
\begin{thm}[Non-outlier eigenvectors]\label{thm_nonoutliereigenvector} Suppose that Assumptions \ref{assu_limit} and \ref{assum_outlier} hold. Fix any  small constant $\tau>0.$ For $\pi_a(i) \notin \mathcal{O}_+, \ i \leq \tau N$ and any deterministic vector $\vb=(v_1, \cdots, v_N)^* \in \mathbb{R}^N,$ we have that 
\begin{equation}\label{eq_nonoutliereigenvector}
\left| \langle \vb, \widehat{\ub}_{\pi_a(i)} \rangle \right|^2 \prec \sum_{j=1}^N \frac{|v_j|^2}{N(\kappa_i+|\widehat{a}_j-\Omega_{\beta}(E_{+})|^2)}. 
\end{equation}
Similarly, for the right singular vectors, we have 
\begin{equation*}
\left| \langle \vb, \widehat{\vb}_{\pi_b(j)} \rangle \right|^2 \prec \sum_{i=1}^N \frac{|v_i|^2}{N (\kappa_j+|\widehat{b}_i-\Omega_{\alpha}(E_{+})|^2)}. 
\end{equation*}
\end{thm}

Theorem \ref{thm_nonoutliereigenvector} establishes the results for the non-outlier eigenvectors. Especially, if the spikes are well separated from the critical values at a distance of constant order, (\ref{eq_nonoutliereigenvector}) implies that the non-outlier eigenvectors are completely delocalized in the direction of the associated spiked eigenvectors. {More specifically, assuming that for the spike $\widehat{a}_l \geq \Omega_{\beta}(E_{+})+\varsigma,$ for some fixed constant $\varsigma>0$ and $1 \leq l \leq r,$  we consider the direction of the associated spiked eigenvector, i.e., $\mathbf{v}=\mathbf{e}_l.$ In such a setting, for $\pi_a(i) \notin \mathcal{O}_+,$ the result of (\ref{eq_nonoutliereigenvector}) reads 
\begin{equation*}
\left| \langle \mathbf{e}_l, \widehat{\mathbf{u}}_{\pi_a(i)} \rangle \right|^2 \prec N^{-1},   
\end{equation*}  
which indicates that the eigenvector $\widehat{\mathbf{u}}_{\pi_a(i)}$ is completely delocalized when projected on the eigenvector associated with $\widehat{a}_l.$ Similar results hold for the right singular vectors. 
}

\section{Preliminary results and proof strategy}\label{sec_toolsandproofs}
In this section, we collect and prove some preliminary results which will be used in our technical proof. For any spectral parameter $z=E+\ii\eta\in\C_{+}$, we define
		\beq\label{eq_defnkappa}
		\kappa\equiv\kappa(z):=\absv{E-E_{+}},
	\eeq
	where  $E_{+}$ is the  rightmost edge of $\mu_{\alpha}\boxtimes\mu_{\beta}$ given in (\ref{eq_edge}). For each $0\leq a< b$ and $0<\tau<\min\{\frac{E_{+}-E_{-}}{2},1\}$, we define the following set of spectral parameter $z$ by
	\beq \label{eq_fundementalset}
	{\caT_{\tau}(a,b)}\deq \{z=E+\ii\eta\in\C_{+}:E_+-\tau \leq E \leq \tau^{-1}, a<\eta< b\}.
	\eeq
	Furthermore, for $0<\xi<1$, we let 
	\beq\label{eq_eltalgamma}
		\eta_{L}\equiv\eta_{L}(\xi)\deq N^{-1+\xi},
	\eeq and let $\eta_{U}>1$ be a large $N$-independent constant. Finally, we define another  domain of spectral parameters
\beq \label{eq_fundementalsetoutlier}
	{ \caT_\tau(\eta_U)} \deq \{z=E+\ii\eta\in\C : E_++N^{-2/3+\tau} \leq E \leq \tau^{-1}, \ |\eta|<\eta_U \}.
	\eeq

\subsection{Properties of free multiplicative convolution and subordination functions}\label{sec_subordinationpre}

In this subsection, we collect some important properties regarding $\mu_{\alpha} \boxtimes \mu_{\beta}, \mu_{A} \boxtimes \mu_{B}$ and their associated subordination functions. In the following lemma, we summarize the results regarding the properties of $\mu_{A} \boxtimes \mu_B$ on the spectral domains as in (\ref{eq_fundementalset}) and (\ref{eq_fundementalsetoutlier}).

\begin{lem}\label{prop:stabN}
	Suppose Assumption \ref{assu_limit} holds. Fix some small constant $\tau>0,$ denote
	\begin{equation}\label{Dnot}
	{\mathcal{T} \equiv \mathcal{T}(\tau, \eta_L, \eta_U):=\mathcal{T}_{\tau}(\eta_L, \eta_U) \cup \mathcal{T}_{\tau}(\eta_U). }
    \end{equation}	 
Then for sufficiently large $N$, the following hold uniformly over $z\in {\caT}$:
	\begin{itemize}
		\item[(1)] We have
		\begin{align*}
		\min_i\absv{a_{i}-\Omega_{B}(z)}&\gtrsim 1, &\min_i \absv{b_{i}-\Omega_{A}(z)}&\gtrsim 1,&
		 \absv{\Omega_{A}(z)}&\sim 1,& \absv{\Omega_{B}(z)}&\sim 1.
		\end{align*}
		
		\item[(2)] For $\kappa$ defined in (\ref{eq_defnkappa}), we have 
		\beqs
		\im m_{\mu_{A}\boxtimes\mu_{B}}(z)\sim\left\{
		\begin{array}{lcl}
			\sqrt{\kappa+\eta}, &\text{if } \ E\in\supp\mu_{A}\boxtimes\mu_{B},\\
			\dfrac{\eta}{\sqrt{\kappa+\eta}}, &\text{if } \ E\notin\supp\mu_{A}\boxtimes\mu_{B}.
		\end{array}
		\right.
		\eeqs
		
		\item[(3)] For the derivatives of $\Omega_{A}$ and $\Omega_{B}$, we have
		\begin{align}\label{eq_derivativecontrol}
		\absv{\Omega_{A}'(z)}&\sim \frac{1}{\sqrt{|z-E_+|}}\sim
		\absv{\Omega_{B}'(z)} , &
		\absv{\Omega_{A}^{''}(z)}&\sim \frac{1}{|z-E_+|^{3/2}}\sim
		\absv{\Omega_{B}^{''}(z)}.
		\end{align}
		
		\item[(4)] For all fixed $\epsilon>0$, we have 
		\begin{equation}\label{eq_closenessofomega}
		|\Omega_A(z)-\Omega_\alpha(z)|+|\Omega_B(z)-\Omega_\beta(z)| \lesssim \frac{N^{-1+\epsilon}}{\sqrt{\kappa+\eta}}.
        \end{equation}	 
	\end{itemize}
\end{lem}
\begin{proof}
See Proposition 2.11 and its proof in \cite{DJ1}. 
\end{proof}


%
%

\subsection{Local laws for free multiplication of random matrices}
In this subsection, we prove the local laws  which are the key ingredients for our technical proof. We start with introducing some notations. Let $U\equiv U_{N}$ be a random unitary or orthogonal matrix, Haar distributed on the $N$-dimensional unitary group $U(N)$ or orthogonal group $O(N)$. Define $\wt{A}\deq U\adj AU$, $\wt{B}\deq UBU\adj$, and 
\beq\label{defn_eq_matrices}
	H\deq AUBU\adj,	\quad	\caH\deq U\adj AU B,	\quad	
	\wt{H}\deq A^{1/2}\wt{B}A^{1/2}, \AND \wt{\caH}\deq B^{1/2}\wt{A}B^{1/2}. 
\eeq 
Note that we only need to consider diagonal 	matrices $A$ and $B$ since $U$ is a Haar random unitary or orthogonal  matrix. Moreover, $\wt{H}$ and $\wt{\caH}$ are Hermitian random matrices. 
		
Since $H$, $\caH$, $\wt{H}$ and $\wt{\caH}$ have the same eigenvalues, we denote the common eigenvalues as $\lambda_1 \geq \lambda_2 \geq \cdots \geq\lambda_N$ in what follows. 
For $z\in\C_{+}:=\{z \in \mathbb{C}: \im z >0\}$, we define the \emph{resolvents} of the above random matrices as follows
	\beq\label{defn_greenfunctions} 
		G(z)\deq (H-zI)^{-1},\quad \caG(z)\deq (\caH-zI)^{-1},\quad \wt{G}(z)\deq (\wt{H}-zI)^{-1},\quad \wt{\caG}(z)\deq (\wt{\caH}-zI)^{-1}.
	\eeq
In the rest of the paper, we usually omit the dependence of $z$ and simply write $G, \caG, \wt{G}$ and $\wt{\caG}.$	
Let $\mu_{H}$ be the ESD of $H$ and $m_H(z)$ be the associated Stieltjes transform.  Since $H,\caH,\wt{H}$ and $\wt{\caH}$ are similar to each other, we have that $m_H(z)=\tr G=\tr \caG =\tr \wt{G}=\tr \wt{\caG}.$

 {Our proof relies on the following linearization crucially.}  
For $z \in \C_{+}$, denote $Y=A^{1/2}U B^{1/2}$ and $\Hb \equiv \Hb(z)$  as 
\begin{equation}\label{eq_BIGH}
\Hb(z) :=
\begin{pmatrix}
0 & z^{1/2} Y \\
z^{1/2}Y^* & 0
\end{pmatrix},
\end{equation}
and $\Gb(z)=(\Hb-z)^{-1}.$
By Schur's complement, it is easy to see that 
\begin{equation}\label{eq_schurcomplement}
\Gb(z)=
\begin{pmatrix}
\wt{G}(z) & z^{-1/2} \wt{G}(z) Y \\
z^{-1/2}Y^* \wt{G}(z) &  \wt{\mathcal{G}}(z)
\end{pmatrix},
\end{equation}
where we recall the definitions in (\ref{defn_greenfunctions}).
To simplify notations, we define the index sets 
\begin{equation*}
\mathcal{I}_1:=\llbra1, N\rrbra,  \  \mathcal{I}_2:=\llbra N+1, 2N\rrbra, \ \mathcal{I}:= \mathcal{I}_1 \cup \mathcal{I}_2.
\end{equation*}
Then we relabel the indices of the matrices according to 
\begin{equation*}
U=(U_{i \mu}: i \in \mathcal{I}_1, \mu \in \mathcal{I}_2), \  A=(A_{ij}: i,j \in \mathcal{I}_1), \  B=(B_{\mu \nu}: \mu, \nu \in \mathcal{I}_2). 
\end{equation*}
In the proof of the paper, we will consistently use the latin letters $i,j \in \mathcal{I}_1$ and greek letters $\mu, \nu \in \mathcal{I}_2$.  We denote the $2N \times 2N$ diagonal matrix $\Theta \equiv \Theta(z)$ by letting 
\begin{equation}\label{eq_Bigthetadefinition}
\Theta_{ii}=\frac{1}{z} \frac{\Omega_B(z)}{a_i-\Omega_B(z)}, \ \Theta_{\mu \mu}=\frac{1}{z} \frac{\Omega_A(z)}{b_\mu-\Omega_A(z)}. 
\end{equation} 
 We have the following controls for the resolvent $\Gb(z)$ uniformly in $z \in {\mathcal{T}_{\tau}(\eta_L, \eta_U)}.$  {Denote $\gamma_j$ as the $j$-th $N$-quantile (or typical location) of $\mu_{\alpha} \boxtimes \mu_{\beta}$ such that 
\begin{equation}\label{eq_defngammajj}
\int_{\gamma_j}^\infty d \mu_{\alpha} \boxtimes \mu_{\beta}(x)=\frac{j}{N}.
\end{equation}
Similarly, we denote $\gamma_j^*$ to be the $j$-th $N$-quantile of $\mu_A \boxtimes \mu_B.$ }
\begin{thm}[Local laws]\label{prop_linearlocallaw}
Suppose Assumption \ref{assu_limit} holds. Let $\tau$ and $\xi$ be fixed small positive constants. We have that 
\begin{equation}\label{locallaw_edgeentrywise}
\sup_{1 \leq k ,l \leq 2N}|(\Gb(z)-\Theta(z))_{kl}| \prec \sqrt{\frac{\im m_{\mu_A \boxtimes \mu_B}}{N \eta}}+\frac{1}{N \eta},
\end{equation}
and 
\begin{equation}\label{locallaw_edgeaveraged}
|m_H(z)-m_{\mu_A \boxtimes \mu_B}(z)| \prec \frac{1}{N \eta},
\end{equation}
hold uniformly in $z \in {\mathcal{T}_{\tau}(\eta_L, \eta_U)}.$ Moreover, far away from the spectrum, for  $z \in {\mathcal{T}_{\tau}(\eta_U)}$ uniformly,  we have that 
\begin{equation}\label{locallaw_outlierentrywise}
\sup_{1 \leq k ,l \leq 2N}|(\Gb(z)-\Theta(z))_{kl}| \prec N^{-1/2}(\kappa+\eta)^{-1/4},
\end{equation}
and 
\begin{equation}\label{locallaw_outlieraveraged}
|m_H(z)-m_{\mu_A \boxtimes \mu_B}(z)| \prec \frac{1}{N(\kappa+\eta)}. 
\end{equation}
\end{thm}
\begin{proof}
When {$z \in \mathcal{T}_{\tau}(\eta_L, \eta_U),$} (\ref{locallaw_edgeaveraged}) has been proved in Theorem 2.13 of \cite{DJ1}. Moreover, (\ref{locallaw_edgeentrywise}) has also been established therein with a slightly different form and it can be proved following the lines of the proof of \cite[Theorem 2.13]{DJ1}; see Remark 2.14 therein for more details.

When {$z \in \mathcal{T}_{\tau}(\eta_U),$} (\ref{locallaw_outlierentrywise}) follows from \eqref{eq_localquadtratic}, Lemma \ref{prop:stabN}, and (\ref{locallaw_edgeentrywise}). The calculation is standard in the random matrix theory literature; for instance, see Theorem 3.12 and its proof in \cite{alex2014}. We omit the details here.  Finally, we prove (\ref{locallaw_outlieraveraged}). The discussion is similar to that of equation (2.19) of \cite{BEC} and we only sketch the key points.  We fix an arbitrary chosen $\epsilon\in(0,\tau/100)$ and consider the event $\Xi$ on which we have
	\begin{align}\label{eq:locallawout_Xi}
		&\sup_{z\in\caD_{\tau}(\eta_{L},\eta_{U})}\eta\absv{m_{H}(z)-m_{\mu_{A}\boxtimes\mu_{B}}(z)}\leq N^{-1+\epsilon/2},&
		&\max_{1\leq i\leq N/3}i^{1/3}\absv{\lambda_{i}-\gamma_{i}}\leq N^{-2/3+\epsilon},
	\end{align}
	where in the proof we choose $\eta_{L}=N^{-1+\epsilon}$. By  (\ref{locallaw_edgeentrywise}), (\ref{locallaw_edgeaveraged}) and Lemma \ref{thm_rigidity} below, we have $\prob{\Xi}\geq 1- N^{-D}$ for any large $D>0$. For all $z_{0}=E_{0}+\ii\eta_{0}\in {\caT_{\tau}(\eta_{U})}$ with $4\eta_{0}\leq\kappa_{0}=\absv{E_{0}-E_{+}}$, we consider a counter-clockwise square contour $\caC(z_{0})$ with side length $\kappa_{0}$ and (bary)center $z_{0}$. Then, on the event $\Xi,$ by Cauchy's theorem, we have
	\beqs
		m_{H}(z_{0})-m_{\mu_{A}\boxtimes\mu_{B}}(z_{0})
		=\left(\int_{\caC_{>}(z_{0})}+\int_{\caC_{\leq}(z_{0})}\right)\frac{m_{H}(z)-m_{\mu_{A}\boxtimes\mu_{B}}(z)}{z-z_{0}}\dd z,
	\eeqs
	where $\caC_{>}(z_{0})=\caC(z_{0})\cap\{z:\absv{\im z}>\eta_{L}\}$ and $\caC_{\leq}(z_{0})=\caC(z_{0})\cap\{z:\absv{\im z}\leq \eta_{L}\}$. On the contour $\caC_{>}(z_{0}),$ we use the first bound in \eqref{eq:locallawout_Xi} to get that for some constant $C>0$
	\begin{align}\label{eq:contour_in}
		&\Absv{\int_{\caC_{>}(z_{0})}\frac{m_{H}(z)-m_{\mu_{A}\boxtimes\mu_{B}}(z)}{z-z_{0}}\dd z}
		\leq N^{-1+\epsilon/2}\frac{2}{\kappa_{0}}\Absv{\int_{\caC_{>}(z_{0})}\frac{1}{\im z}\dd z}\\
		& \leq N^{-1+\epsilon/2}\frac{4}{\kappa_{0}}\left(\frac{\kappa_{0}}{\eta_{0}+\kappa_{0}/2}+\log\left(\frac{\eta_{0}+\kappa_{0}/2}{\eta_{L}}\right)\right)\leq \frac{C}{\kappa_{0}+\eta_{0}}N^{-1+\epsilon/2}\log N\leq \frac{C}{\kappa_{0}+\eta_{0}}N^{-1+\epsilon}. \nonumber
	\end{align}	
On the other hand, for $z$ on the other contour $\caC_{\leq}(z_{0}),$ we use
	\beqs
		\absv{m_{H}(z)}\leq \frac{1}{N}\sum_{i}\frac{1}{E-\lambda_{i}}\leq \frac{1}{N}\sum_{i\leq N/3} \frac{1}{E-\gamma_{i}-i^{-1/3}N^{-2/3+\epsilon}}+\frac{1}{N}\sum_{i>N/3}\frac{1}{E-\gamma_{N/3}-3N^{-1+\epsilon}}\leq C,
	\eeqs
	where in the second step we used the second bound in \eqref{eq:locallawout_Xi} and in the third step we used the fact that $E_+-\gamma_i \sim i^{2/3} N^{-2/3}$. Following the same argument and  using Lemma \ref{thm_rigidity} below, we get $\absv{m_{\mu_{A}\boxtimes\mu_{B}}(z)}\leq C$. Using the above bounds, we get
	\beq\label{eq:contour_out}
		\Absv{\int_{\caC_{\leq}(z_{0})}\frac{m_{H}(z)-m_{\mu_{A}\boxtimes\mu_{B}}(z)}{z-z_{0}}\dd z}\leq \frac{C}{\kappa_{0}}\eta_{L}=C\frac{N^{\epsilon}}{N\kappa_{0}}\leq \frac{C}{\kappa_{0}+\eta_{0}}N^{-1+\epsilon}.
	\eeq
	Combining \eqref{eq:contour_in} and \eqref{eq:contour_out}, we conclude our proof.

\end{proof}

Finally, we collect two important consequences of (\ref{locallaw_edgeentrywise}) and (\ref{locallaw_edgeaveraged}).   Recall (\ref{eq_defngammajj}) and $\lambda_1 \geq \lambda_2 \geq \cdots \geq  \lambda_N$ are the eigenvalues of $AUBU^*.$
\begin{lem}[Spectral rigidity near the upper edge] \label{thm_rigidity} Suppose Assumption \ref{assu_limit} holds true. For any small constant $0<c<1/2,$ we have that for all $1 \leq i \leq cN,$
\begin{equation*}
|\lambda_i-\gamma_i^*| \prec i^{-1/3}N^{-2/3}. 
\end{equation*}  
Moreover, the same conclusion holds if $\gamma_i^*$ is replaced with $\gamma_i.$
\end{lem}
\begin{proof}
See Theorem 2.15 of \cite{DJ1}.
\end{proof}
Denote the singular value decomposition (SVD) of $Y=A^{1/2} UB^{1/2}$ as 
\begin{equation*}
Y=\sum_{k=1}^N \sqrt{\lambda_i} \ub_k \vb_k^*,
\end{equation*}
where $\{\ub\}_k$ and $\{\vb_k\}$ are the left and right singular vectors of $Y,$ respectively. 
\begin{lem}[Delocalization of the singular vectors] \label{thm_delocalization}
Suppose Assumption \ref{assu_limit} holds true. For any small constant $0<c<1/2,$ we have that for all $1 \leq k \leq cN,$
\begin{equation*}
\max_i|\ub_k(i)|^2+\max_\mu|\vb_k(\mu)|^2 \prec \frac{1}{N}. 
\end{equation*}
\end{lem}
\begin{proof}
See Theorem 2.16 of \cite{DJ1}. 
\end{proof}

\subsection{Proof strategy}\label{sec_proofstartegy}
In this subsection, we summarize the proof strategies. We focus on explaining how to adapt the techniques of the spiked covariance matrix model \cite{Bloemendal2016} and separable  
covariance matrix model \cite{DYaos} to obtain the results for our spiked multiplicative model. Analogous strategies have also been used to study the deformed Wigner matrix in \cite{KY13,MR3262497}, general spiked sample covariance matrix in \cite{DINGRMTA} and signal-plus-noise matrix in \cite{MR4036038}.  

First, we discuss how to prove Theorem \ref{thm_outlier}. For the outlier eigenvalues, their locations satisfy a master equation in terms of the resolvent as summarized in Lemma \ref{lem_evmaster} below. {Recall (\ref{eq_spikes}). } Denote 
\begin{equation}\label{eq_defnudefnd}
\mathbf{U}=
\begin{pmatrix}
\mathbf{E}_r & 0 \\
0 & \mathbf{E}_s
\end{pmatrix}, \ \bm{\mathcal{D}}=
\begin{pmatrix}
D^a (D^a+1)^{-1} & 0 \\
0 & D^b (D^b+1)^{-1}
\end{pmatrix},
\end{equation}
where $\mathbf{E}_r=(\mathbf{e}_1, \cdots, \mathbf{e}_r)$, \ $\mathbf{E}_s=(\mathbf{e}_1, \cdots, \mathbf{e}_s)$, $D^a=\operatorname{diag}(d_1^a, \cdots, d_r^a)$ and $D^b=\operatorname{diag}(d_1^b, \cdots, d_s^b).$ Recall (\ref{matrixnotation}) and (\ref{eq_schurcomplement}). { Recall $\Hb(z)$ in (\ref{eq_BIGH}). Define $\wh \Hb \equiv \wh \Hb(z)$  as 
\begin{equation}\label{eq_wthdecomposition}
\wh \Hb(z):=\mathbf{P} \Hb(z) \mathbf{P} =
\begin{pmatrix}
0 & z^{1/2} \wh Y \\
z^{1/2}\wh Y^* & 0
\end{pmatrix}, \
\mathbf{P}=
\begin{pmatrix}
(1+D^a)^{1/2} & 0\\
0 & (1+D^b)^{1/2}
\end{pmatrix}.
\end{equation}
Correspondingly, we denote $\wh{\Gb}(z)=(\wh{\Hb}(z)-z)^{-1}.$ }

\begin{lem}\label{lem_evmaster}
If $x \neq 0$ is not an eigenvalue of $\mathcal{Q}_1,$ then it is an eigenvalue of $\wh{\mathcal{Q}}_1$ if and only if 
\begin{equation*}
\det (\bm{\mathcal{D}}^{-1}+x\mathbf{U}^* \mathbf{G}(x) \mathbf{U} )=0. 
\end{equation*} 
\end{lem}
\begin{proof}
{By the linearization construction in (\ref{eq_wthdecomposition}), we find that the eigenvalues of $z^{-1/2} \widehat{\mathbf{H}}$ are given by 
\begin{equation*}
\pm \sqrt{\lambda_1(\widehat{\mathcal{Q}}_1)}, \  \pm \sqrt{\lambda_2(\widehat{\mathcal{Q}}_1)}, \cdots, \pm \sqrt{\lambda_N(\widehat{\mathcal{Q}}_1)}. 
\end{equation*}
Hence, it is clear that $x>0$ is an eigenvalue of $\widehat{\mathcal{Q}}_1$ if and only if 
\begin{equation*}
\det(\widehat{\mathbf{H}}(x)-x)=0. 
\end{equation*}
Since $\mathbf{P}$ and  $\bm{\mathcal{D}}$ are always invertible, $x \neq 0$ is an eigenvalue of $\widehat{\mathbf{H}}=\mathbf{P} \mathbf{H} \mathbf{P}$ if and only if    
\begin{align*}
0&=\det(\mathbf{PHP}-x)=\det\Big(\mathbf{P}(\mathbf{H}-\mathbf{P}^{-2}x)\mathbf{P}\Big)=\det(\mathbf{P}^2) \det(\mathbf{G}(x))\det \left(1+x\mathbf{G}(x)(1-\mathbf{P}^{-2}) \right) \\
&=\det(\mathbf{P}^2) \det(\mathbf{G}(x))\det \left(1+x\mathbf{G}(x) \mathbf{U} \bm{\mathcal{D}} \mathbf{U}^* \right)=\det(\mathbf{P}^2) \det(\mathbf{G}(x)) \det \Big(1+x \mathbf{U}^* \mathbf{G}(x) \mathbf{U} \bm{\mathcal{D}}  \Big) \nonumber \\
& =\det(\mathbf{P}^2) \det(\mathbf{G}(x))\det(\bm{\mathcal{D}})\det(\bm{\mathcal{D}}^{-1}+x \mathbf{U}^* \mathbf{G}(x) \mathbf{U}),
\end{align*}
where in the second step we used $\det(1+AB)=\det(1+BA)$. We can then conclude that proof as $x$ is not an eigenvalue of $\mathcal{Q}_1.$ 
}
\end{proof} 
Heuristically, by Theorem \ref{prop_linearlocallaw} and Lemma \ref{lem_evmaster}, an outlier location $x > E_+$ should satisfy the condition that 
\begin{equation}\label{eq_determinantexplicitform}
\prod_{i=1}^r \left( \frac{d_i^a+1}{d_i^a}+\frac{\Omega_B(x)}{a_i-\Omega_B(x)} \right)  \prod_{j=1}^s \left( \frac{d_j^b+1}{d_j^b}+\frac{\Omega_A(x)}{b_j-\Omega_A(x)} \right)=0.
\end{equation} 
Therefore, solving (\ref{eq_determinantexplicitform}) will yield the locations of the outlier eigenvalues. Moreover,
according to \eqref{eq_localquadtratic} and (\ref{eq_closenessofomega}), we find that
\begin{equation*}
\frac{d_i^a+1}{d_i^a}+\frac{\Omega_B(x)}{a_i-\Omega_B(x)}=0\quad\text{for some $x>E_{+}$}\quad\Longleftrightarrow\quad \frac{d_i^a+1}{d_i^a}+\frac{\Omega_{\beta}(E_{+})}{a_i-\Omega_{\beta}(E_{+})}<0,
\end{equation*}
which in turn is equivalent to $\wh{a}_i>\Omega_{\beta}(E_{+}).$ Similar calculation holds for $\wh{b}_j, 1 \leq j \leq s. $ Furthermore,  to obtain the convergence rates for the outlier eigenvalues, it suffices to apply the strategy developed in \cite{Bloemendal2016}. The proof consists of the following three steps: (1). Construct the permissible regions which contain all the eigenvalues of $\widehat{\mathcal{Q}}_1$ with high probability (c.f. Lemma \ref{lem_permissibleregion}). This step enables us to find the a union of sets in which all the eigenvalues including the outliers lie. (2). Apply a counting argument using Rouch{\' e}'s theorem to a special case where the algebraic multiplicities of the spikes are all unity to show that each connected component of the permissible region contains the correct number of the eigenvalues of $\widehat{\mathcal{Q}}_1.$ (3). Use a continuous interpolation argument to extend Step (2) to the general case with a careful discussion on the gaps in the permissible region for the general setting. We point out that Steps (2) and (3) are quite standard, for example see \cite{KY13, Bloemendal2016,DYaos}, and we will focus on explaining the differences lying in Step (1) in the proof. For the extremal non-outlier eigenvalues, the proof makes use of the results in Steps (2) and (3). The details will be provided in Section \ref{sec_thmoutlier}. 

Second, we explain how to justify Theorem \ref{thm_eigenvaluesticking}. The proof again consists of three steps similar to those described earlier. In the first step, we construct the permissible regions for the eigenvalues; see Lemma \ref{lem_nonoutlierpermissible} below for more details. In the second step, we prove our results to a special case (c.f. Lemma \ref{lem_nonoutcounting}) and in the third step, we prove Theorem \ref{thm_eigenvaluesticking} using a continuity argument. Compared to the outlier eigenvalues, the second step is more complicated since Rouch{\' e}'s theorem is not applicable. Instead, we need to employ a perturbation argument. The details can be found in Section \ref{sec_proofsticking}.

Third, we discuss how to prove Theorem \ref{thm_outlieereigenvector}. Recall (\ref{eq_wthdecomposition}) and $\widehat{\mathbf{G}}(z)$ defined around it. We first provide some useful expressions.  By a discussion similar to (\ref{eq_schurcomplement}) and the singular value decomposition (SVD) of $Y$, we have that 
\begin{align}\label{eq_linearizationspectral}
\wh{\Gb}_{ij}=\sum_{k=1}^N \frac{\wh{\mathbf{u}}_k(i) \wh{\ub}_k^*(j)}{\hlam_k-z}, \  \wh{\Gb}_{\mu \nu}=\sum_{k=1}^N \frac{\wh{\vb}_k(\mu) \wh{\vb}_k^*(\nu)}{\hlam_k-z}, 
\end{align}
\begin{align*}
\wh{\Gb}_{i \mu}=\frac{1}{\sqrt{z}} \sum_{k=1}^N \frac{\sqrt{\hlam_k} \wh{\ub}_k(i)\wh{\vb}_k^*(\mu)  }{\hlam_k-z}, \ \wh{\Gb}_{\mu i}=\frac{1}{\sqrt{z}} \sum_{k=1}^N \frac{\sqrt{\hlam_k} \wh{\vb}_k(\mu)\wh{\ub}_k^*(i)  }{\hlam_k-z}.
\end{align*}
Following the strategy of \cite{Bloemendal2016, DYaos}, the starting point is an integral representation. Specifically, by (\ref{eq_linearizationspectral}), Lemma \ref{lem_contour} below and Cauchy's integral formula, for some properly chosen contour $\Gamma$ around the outliers,  we have that 
\begin{equation}\label{eq_intergralrepresentation}
\langle \mathbf{e}_i, \mathcal{P}_S \mathbf{e}_j \rangle=-\frac{1}{2 \pi \ii} \oint_{\Omega_B^{-1}(\Gamma)} \langle \bm{e}_i, \wh{\Gb}(z) \bm{e}_j \rangle \dd z, 
\end{equation}
where $\bm{e}_i$ and $\bm{e}_j$ are the natural embeddings of $\mathbf{e}_i$ and $\mathbf{e}_j$ in $\mathbb{C}^{2N}.$ Next, we provide an identity for $\Ub^* \wh{\Gb}(z) \Ub$ in terms of $\Ub^* \Gb(z) \Ub$ (recall (\ref{eq_defnudefnd})). For the matrices $\mathcal{A}, \mathcal{S}, \mathcal{B}$, and $\mathcal{E}$ of conformable dimensions, by Woodbury matrix identity, we have 
\begin{equation}\label{eq_woodbury}
(\mathcal{A}+\mathcal{S} \mathcal{B} \mathcal{E})^{-1}=\mathcal{A}^{-1}-\mathcal{A}^{-1} \mathcal{S} (\mathcal{B}^{-1}+\mathcal{E} \mathcal{A}^{-1} \mathcal{S})^{-1} \mathcal{E} \mathcal{A}^{-1},
\end{equation} 
as long as all the operations are legitimate. Moreover, when $\mathcal{A}+\mathcal{B}$ is non-singular, we have that 
\begin{equation}\label{eq_hua}
\mathcal{A}-\mathcal{A}(\mathcal{A}+\mathcal{B})^{-1} \mathcal{A}=\mathcal{B}-\mathcal{B}(\mathcal{A}+\mathcal{B})^{-1} \mathcal{B}.
\end{equation}
By (\ref{eq_wthdecomposition}), (\ref{eq_linearizationspectral}) and the matrix identities (\ref{eq_woodbury}) and (\ref{eq_hua}), we have that 
\begin{align}\label{eq_keyexpansionev}
\Ub^* \wh{\Gb}(z) \Ub&=\Ub^* \mathbf{P}^{-1} \left( \Hb-z+z(I-\mathbf{P}^{-2}) \right)^{-1} \mathbf{P}^{-1} \Ub=\mathbf{U}^*\mathbf{P}^{-1}(\Gb^{-1}(z)+z\Ub \bm{\mathcal{D}} \Ub^*)^{-1} \mathbf{P}^{-1} \Ub \nonumber \\
&= \Ub^* \mathbf{P}^{-1} \left[ \Gb(z)-z\Gb(z) \Ub \frac{1}{\bm{\mathcal{D}}^{-1}+z \Ub^* \Gb(z) \Ub} \Ub^* \Gb(z) \right] \mathbf{P}^{-1} \Ub \nonumber \\
&=(I-\bm{\caD})^{1/2} \left[ \Ub^* \Gb(z) \Ub-z \Ub^* \Gb(z) \Ub \frac{1}{\bm{\mathcal{D}}^{-1}+z\Ub^* \Gb(z) \Ub} \Ub^* \Gb(z) \Ub \right] (I-\bm{\caD})^{1/2} \nonumber \\
&= \frac{1}{z}(I-\bm{\caD})^{1/2}  \left[{\bm{\mathcal{D}}^{-1}}-\bm{\mathcal{D}}^{-1} \frac{1}{\bm{\mathcal{D}}^{-1}+z\Ub^* \Gb(z) \Ub} \bm{\mathcal{D}}^{-1} \right] (I-\bm{\caD})^{1/2},
\end{align}
where we used $(I-\bm\caD)=\mathbf{P}^{-2}$ and $\mathbf{P}$ is defined in (\ref{eq_wthdecomposition}). On one hand, we can expand (\ref{eq_intergralrepresentation}) using (\ref{eq_keyexpansionev}). On the other hand, since (\ref{eq_keyexpansionev}) can be well-estimated  using the local laws Theorem \ref{prop_linearlocallaw}, the limit of (\ref{eq_intergralrepresentation}) can be calculated using the residue theorem.  To obtain the convergence rates, we need to apply a high order resolvent expansion for (\ref{eq_keyexpansionev}) as in (\ref{eq_resolventexpansion}). In the actual proof, we split it into two steps due to some technicalities. In the first step, we prove the results under a slightly stronger assumption than Assumption \ref{assum_outlier} that 
\begin{equation}\label{condition_nonoverone}
	\wh{a}_i-\Omega_{\beta}(E_{+}) \geq N^{-1/3+\tau_1}, \  \wh{b}_j-\Omega_{\alpha}(E_{+}) \geq N^{-1/3+\tau_1}, \ i, j \in \mathcal{O}^+,
\end{equation}
where $\tau_1>0$ is some small positive constant. The details are provided in Proposition \ref{prop_maineigenvector} and proved in  Section \ref{sec_prove36withconstrain}. Then in the second step, we remove (\ref{condition_nonoverone}) using a finer estimate to conclude the proof of Theorem \ref{thm_outlieereigenvector}. The details are provided in Section \ref{sec_prove36withoutconstrain}.

Finally, we explain how to handle the non-outlier eigenvectors in Theorem \ref{thm_nonoutliereigenvector}. In this setting, the residual calculus fails since the contour representation in (\ref{eq_intergralrepresentation}) is invalid. Instead, we utilize the inequality
\begin{equation}\label{nonoutlier_keyexpansion}
\left| \langle \vb, \widehat{\ub}_k  \rangle \right|^2 \leq \eta \left( \vb^* \sum_{k=1}^N \frac{\eta \widehat{\ub}_k \widehat{\ub}_k^*}{|\widehat{\lambda}_k-z_k|^2} \vb \right)=\eta \im \vb^* \widehat{G}(z_k) \vb,
\end{equation}
where $\widehat{G}(z)=(\widehat{\mathcal{Q}}_1-z)^{-1}$ and $z_k=\widehat{\lambda}_k+\ii \eta.$ Consequently, it suffices to obtain an accurate estimate for $\im \vb^* \widehat{G}(z_k) \vb.$  We again use the local law Theorem \ref{prop_linearlocallaw} combining with a resolvent expansion to establish the delocalization bounds. The details can be found in Section \ref{sec_proofofnonoutliereigenvector}.

\begin{rem}\label{rem_localize}
We remark that the results regarding outlier eigenvalues and eigenvectors in  Theorems \ref{thm_outlier} and \ref{thm_outlieereigenvector} remain true even if we loosen Assumption \ref{assu_limit} in the following sense: (1). 
(iii) and (v) can be replaced by 
	\beqs
		\rho_{\alpha(\beta)}(x)\sim (E^{\alpha(\beta)}_{+}-x)^{t_{+}^{\alpha(\beta)}},\,\, x\in[E_{+}^{\alpha(\beta)}-c,E_{+}^{\alpha(\beta)}]\AND \sup\supp\mu_{A(B)}\to E^{\alpha(\beta)}_{+},
	\eeqs 
	for a small constant $c>0$, and (2). the error in (\ref{eq_epsilondefinition}) can be replaced by $N^{-2/3-\epsilon}.$ However, under such weaker assumptions, the results in Theorems \ref{thm_eigenvaluesticking} and \ref{thm_nonoutliereigenvector} will be weakened, especially (\ref{eq_stickingbound}) and (\ref{eq_nonoutliereigenvector}) only hold up to $i \leq N^{\tau}$ for some small $0<\tau<1.$ Technically, this is because the proof of Theorems \ref{thm_outlier} and \ref{thm_outlieereigenvector} only rely on the square root behavior of $\mu_A \boxtimes \mu_B$ and the local laws near the edges and outside the bulk of the spectrum, i.e., Theorem \ref{prop_linearlocallaw} holds for $\mathcal{T}_\tau(n^{-2/3+\epsilon}, \eta_U)$ and $\mathcal{T}_{\tau}(\eta_U).$ For the square root decay behavior, according to \cite[Equation (A.17)]{DJ1}, 
\begin{equation}\label{eq_squarrootextendsion}
	\operatorname{Im} m_{\mu_{A}\boxtimes\mu_{B}}(z)=\operatorname{Im} m_{\mu_{\alpha}\boxtimes\mu_{\beta}}(z)+\mathrm{O}\left(\frac{\boldsymbol{d}}{\sqrt{|z-E_{+}|}}\right).
\end{equation}
In order to ensure that the error is dominated near the edge by $\operatorname{Im} m_{\mu_{\alpha}\boxtimes\mu_{\beta}}(z)\sim \sqrt{|z-E_{+}|}$ for all $\operatorname{Im}z\gg N^{-2/3}$, we need $\boldsymbol{d}\ll N^{-2/3}$. Armed with the square root decay behavior, the proof only needs the weakened conditions mentioned above. 

In contrast, in order to fully prove Theorems \ref{thm_eigenvaluesticking} and \ref{thm_nonoutliereigenvector}, we need (\ref{eq_epsilondefinition}) so that (\ref{eq_squarrootextendsion}) implies $\operatorname{Im} m_{\mu_{\alpha}\boxtimes\mu_{\beta}}(z)$ is the main term  for all $\operatorname{Im} z \gg \eta_L$.
We also need  the current Theorem \ref{prop_linearlocallaw} which relies on Assumption \ref{assu_limit}; see \cite{Bloemendal2016,KY13} for more details.   
\end{rem}\normalcolor

\section{Proof of Theorems \ref{thm_outlier} and \ref{thm_eigenvaluesticking}: the eigenvalue statistics} \label{sec_proofofspikedmodel}
In this section, we prove the main results of Section \ref{sec_spikedmodel} following the strategies outlined in Section \ref{sec_proofstartegy}. Due to similarity, we focus on explaining the main differences from the counterparts in \cite{Bloemendal2016, DYaos}  and how to adapt their proof strategies.


\subsection{Proof of Theorem \ref{thm_outlier}}\label{sec_thmoutlier}

In this subsection, we prove the results for the outlier and extremal non-outlier eigenvalues. By Theorem \ref{prop_linearlocallaw} and Lemma \ref{thm_rigidity}, for any fixed small constant $\epsilon>0,$ we can choose a high probability event $\Xi \equiv \Xi(\epsilon)$ where the following estimates hold:
\begin{equation}\label{eq_proofinsidelocallaw}
\mathbf{1}(\Xi) \norm{\Ub^*(\Gb(z)-\Theta(z))\Ub} \leq N^{\epsilon/2} \left( \sqrt{\frac{\im m_{\mu_A \boxtimes \mu_B}}{N \eta}}+\frac{1}{N \eta} \right), \ z \in {\mathcal{T}_\tau(\eta_L, \eta_U)} ;
\end{equation}
\begin{equation}\label{eq_proofoutlierlocallaw}
\mathbf{1}(\Xi) \norm{\Ub^* (\Gb(z)-\Theta(z))\Ub} \leq N^{-1/2+\epsilon/2} (\kappa+\eta)^{-1/4}, \ z \in {\mathcal{T}_{\tau}(\eta_U)};
\end{equation}
\begin{equation}\label{eq_proofoutlier}
\mathbf{1}(\Xi)|\lambda_i(\mathcal{Q}_1)-E_+| \leq N^{-2/3+\epsilon}, \ 1 \leq i \leq \varpi. 
\end{equation}
We will restrict our proof to $\Xi$ in what follows and hence the discussion below will be entirely deterministic. We first prepare some notations following the proof of \cite[Theorem 3.6]{DYaos}.
For any fixed constant $\epsilon>0,$ we denote in index sets
\begin{equation}\label{eq_defnsetsab}
\mathcal{O}_{\epsilon}^{(a)}:=\left\{{1 \leq i \leq N}: \wh{a}_i-\Omega_{\beta}(E_{+}) \geq N^{-1/3+\epsilon} \right\}, \ \mathcal{O}_{\epsilon}^{(b)}:=\left\{N+1 \leq \mu \leq 2N, \wh{b}_\mu-\Omega_{\alpha}(E_{+}) \geq N^{-1/3+\epsilon} \right\}.
\end{equation} 
Here and after, we use $\wh{b}_{\mu}:=\wh{b}_{\mu-N}$ for $\mu \in \mathcal{I}_2$. 
Recall that the eigenvalues of $\wh{A}$ and $\wh{B}$ are ordered in the decreasing fashion. By definition, it is easy to see that 
\begin{equation*}
\sup_{\mu \notin \mathcal{O}_{\epsilon}^{(b)}} (\wh{b}_{\mu}-\Omega_{\alpha}(E_{+}))  \lesssim N^{-1/3+\epsilon}, \ \inf_{\mu \in \mathcal{O}_{\epsilon}^{(b)}} (\wh{b}_{\mu}-\Omega_{\alpha}(E_{+})) \gtrsim N^{-1/3+\epsilon}.
\end{equation*} 
%
Moreover, since we are mainly interested in the outlier and extremal non-outlier eigenvalues, we use the convention that $\Omega_{B}^{-1}(\wh{a}_i)=E_+, i \geq r$ and $\Omega_A^{-1}(\wh{b}_\mu)=E_+, \mu \geq N+s.$  Throughout the proof, we will need the following estimate following from  (\ref{eq_localquadtratic}) and (\ref{eq_closenessofomega}), for $i \in \caO_{\epsilon}^{(a)},$
\begin{equation}\label{eq_edgeestimationrough}
\Omega_B^{-1}(\wh{a}_i)-E_+ \sim (\wh{a}_i-\Omega_{\beta}(E_{+}))^2.
\end{equation}
Indeed, when $\Omega_B^{-1}(\wh{a}_i)-E_+ \leq \varsigma_1$ for some sufficiently small constant $0<\varsigma_1<1,$ using (\ref{eq_localquadtratic}), (\ref{eq_closenessofomega}) and the fact $\Omega_B(\cdot)$ is monotone increasing, we readily see that
\begin{equation}\label{eq_smallerregion}
\wh{a}_i \sim \Omega_{\beta}(E_{+})+\gamma \sqrt{\Omega_B^{-1}(\wh{a}_i)-E_+}+N^{-1/2+\epsilon},
\end{equation}
where $\gamma>0$ is some constant. This immediately implies (\ref{eq_edgeestimationrough}) using Assumption \ref{assum_outlier}. On the other hand, when $\Omega^{-1}_B(\wh{a}_i)-E_+ \geq \varsigma_1,$ since $\Omega_B(\cdot)$ is increasing, we obtain that 
\begin{equation*}
\wh{a}_i \geq \Omega_B(E_++\varsigma_1) \sim \Omega_{\beta}(E_{+})+\gamma' \sqrt{\varsigma_1}+N^{-1/2+\epsilon}.
\end{equation*} 
This proves the claim (\ref{eq_edgeestimationrough}). A consequence of \eqref{eq_edgeestimationrough} is that
\begin{equation*}
	\sup_{i \notin \mathcal{O}_{\epsilon}^{(a)}} \Omega_B^{-1}(\wh{a}_i)  \leq  \inf_{\mu \in \mathcal{O}_{\epsilon}^{(b)}} \Omega_A^{-1}(\wh{b}_{\mu})+N^{-2/3+2\epsilon}.
\end{equation*}
Similarly, we have that
\begin{equation*}
\ \sup_{\mu \notin \mathcal{O}_{\epsilon}^{(b)}} \Omega_A^{-1}(\wh{b}_\mu) \leq \inf_{i \in \mathcal{O}_{\epsilon}^{(a)}} \Omega_B^{-1}(\wh{a}_i)+N^{-2/3+2\epsilon}. 
\end{equation*}

An advantage of the above labeling is that the largest outliers of $\wh{\mathcal{Q}}_1$ can be labelled according to $i \in \mathcal{O}_\epsilon^{(a)}$ and $\mu \in \mathcal{O}_{\epsilon}^{(b)}.$ Analogously to (D.9) and (D.10) of \cite{DYaos}, we find that to prove Theorem \ref{thm_outlier}, it suffices to prove that for arbitrarily small constant $\epsilon>0,$ there exists some constant $C>0$ so that 
\begin{equation}\label{eq_reduceproofoutlier}
\mathbf{1}(\Xi) \left| \wh{\lambda}_{\pi_a(i)}-\Omega_B^{-1}(\wh{a}_i) \right| \leq C N^{-1/2+2\epsilon} \Delta_1(\wh{a}_i), \ \ \mathbf{1}(\Xi) \left| \wh{\lambda}_{\pi_b(\mu)}-\Omega_A^{-1}(\wh{b}_\mu) \right| \leq C N^{-1/2+2\epsilon} \Delta_2(\wh{b}_\mu),
\end{equation}
for all $i \in \mathcal{O}_{4 \epsilon}^{(a)}$ and $\mu \in \mathcal{O}_{4\epsilon}^{(b)},$ and \begin{equation}\label{eq_reducebulkvalue}
	\mathbf{1}(\Xi)\left| \wh{\lambda}_{\pi_a(i)}-E_+  \right| \leq CN^{-2/3+12\epsilon}, \  \mathbf{1}(\Xi)\left| \wh{\lambda}_{\pi_b(\mu)}-E_+  \right| \leq CN^{-2/3+12\epsilon},
\end{equation}   
for all $i \in \left\{1,2,\cdots,r \right\} \backslash \mathcal{O}_{4 \epsilon}^{(a)}$ and $\mu \in \left\{N+1,\cdots, N+s \right\} \backslash \mathcal{O}_{4 \epsilon}^{(b)}.$ Here we used the short-hand notations 
\begin{equation}\label{eq_shorthandnotationsdelta}
\Delta_1(\wh{a}_i):=(\wh{a}_i-\Omega_{\beta}(E_{+}))^{1/2}, \ \Delta_2(\wh{b}_\mu):=(\wh{b}_\mu-\Omega_{\alpha}(E_{+}))^{1/2} , 
\end{equation}
and $\pi_b(\mu)\equiv \pi_{b}(\mu-N)$ for $\mu\in \caI_{2}$. 


We now conclude the proof following the four steps outlined in Section \ref{sec_proofstartegy}. We will focus on explaining the first part, Step 1,  since it differs the most from its counterpart of the proof of \cite[Theorem 3.6]{DYaos}, and briefly sketch Steps 2--4.  

\noindent{\bf Step 1:} For each $1 \leq i \leq r^+,$ we define the permissible intervals
\begin{equation*}
\mathrm{I}_i^{(a)}:=\left[\Omega_B^{-1}(\wh{a}_i)-N^{-1/2+\epsilon} \Delta_1(\wh{a}_i), \ \Omega_B^{-1}(\wh{a}_i)+N^{-1/2+\epsilon} \Delta_1(\wh{a}_i) \right].
\end{equation*}
Similarly, for each $1 \leq \mu-N \leq s^+,$ we denote 
\begin{equation*}
\mathrm{I}_\mu^{(b)}:=\left[\Omega_A^{-1}(\wh{b}_\mu)-N^{-1/2+\epsilon} \Delta_2(\wh{b}_\mu), \ \Omega_A^{-1}(\wh{b}_\mu)+N^{-1/2+\epsilon} \Delta_2(\wh{b}_\mu) \right].
\end{equation*}
Then we define
\begin{equation}\label{I0}
\mathrm{I}:=\mathrm{I}_0 \cup \Big(\bigcup_{i \in \mathcal{O}^{(a)}_{\epsilon}}\mathrm{I}_i^{(a)}\Big) \cup  \Big(\bigcup_{\mu \in \mathcal{O}^{(b)}_{\epsilon}}\mathrm{I}_\mu^{(b)}\Big) ,\quad  \mathrm{I}_0:=\left[0, E_+ + N^{-2/3+3\epsilon}\right].
\end{equation}
The main task of this step is to prove the following lemma. 
\begin{lem}\label{lem_permissibleregion} On the event $\Xi$, the complement of $\mathrm{I}$ contains no eigenvalues of $\wh{\mathcal{Q}}_1.$ 
\end{lem}
\begin{proof}
By Lemma \ref{lem_evmaster}, (\ref{eq_proofoutlier}) and (\ref{eq_proofoutlierlocallaw}), we find that $x \notin \mathrm{I}_0$ is an eigenvalue of $\wh{\mathcal{Q}}_1$ if and only if 
\begin{equation}\label{eq_lefthandsidenonsingular}
\mathbf{1}(\Xi)(\bm{\mathcal{D}}^{-1}+x \mathbf{U}^* \Gb(x) \mathbf{U})=\mathbf{1}(\Xi)\left(\bm{\mathcal{D}}^{-1}+x \mathbf{U}^* \Theta(x) \mathbf{U}+\mathrm{O}(\kappa^{-1/4} N^{-1/2+\epsilon/2}) \right),
\end{equation} 
is singular. In light of (\ref{eq_determinantexplicitform}), it suffices to show that if $x \notin \mathrm{I},$ then 
\begin{equation}\label{eq_sufficientconverse}
\min\left\{ \min_{1 \leq i \leq r}\left| \frac{d_i^a+1}{d_i^a}+\frac{\Omega_B(x)}{a_i-\Omega_B(x)} \right|, \min_{1 \leq \mu-N \leq s } \left| \frac{d_\mu^b+1}{d_\mu^b}+\frac{\Omega_A(x)}{b_\mu-\Omega_A(x)} \right| \right\} \gg \kappa^{-1/4} N^{-1/2+\epsilon/2}.
\end{equation}
Indeed, when (\ref{eq_sufficientconverse}) holds, then the matrices on the left-hand side of (\ref{eq_lefthandsidenonsingular}) is non-singular. Note that 
\begin{align}
\frac{d_i^a+1}{d_i^a}+\frac{\Omega_B(x)}{a_i-\Omega_B(x)}=\frac{1}{d_i^a}-\frac{a_i}{\Omega_B(x)-a_i}& =\frac{a_i}{\Omega_B(\Omega_B^{-1}(\wh{a}_i))-a_i}-\frac{a_i}{\Omega_B(x)-a_i} \nonumber \\
&=\mathrm{O} \left( \left| \Omega_B(x)-\Omega_B(\Omega_B^{-1}(\wh{a}_i)) \right| \right), \label{eq_mvtcontrol}
\end{align}
where in the last equality we used (i) of Lemma \ref{prop:stabN}. The rest of the proof is devoted to controlling (\ref{eq_mvtcontrol}) via mean value theorem. 

First, we have that
\begin{equation} \label{eq_mvtdifference}
|x-\Omega_B^{-1}(\wh{a}_i)| \geq N^{-1/2+\epsilon} \Delta_1(\wt{a}_i), \ \text{for all} \ x \notin \mathrm{I}. 
\end{equation}
In fact, when $i \in \mathcal{O}_{\epsilon}^{(a)},$ (\ref{eq_mvtdifference}) holds by definition.  When $i \notin \mathcal{O}_{\epsilon}^{(a)},$ by (\ref{eq_edgeestimationrough}) and the fact that $\Omega_B(x)$ is monotone increasing when $x>E_+$, we have
\begin{equation*}
\Omega_B^{-1}(\wh{a}_i)-E_+ \lesssim  N^{-2/3+2\epsilon} \ll N^{-2/3+3 \epsilon}. 
\end{equation*}  

Now we return to the proof of (\ref{eq_sufficientconverse}). We divide our proof into two cases. If there exists a constant $c>0$ such that $\Omega_B^{-1}(\wh{a}_i) \notin [x-c\kappa, x+c\kappa].$ Since $\Omega_B(\cdot)$ is monotonically increasing on $(E_+, \infty),$ we have that 
\begin{equation*}
|\Omega_B(x)-\Omega_B(\Omega_B^{-1}(\wh{a}_i))| \geq |\Omega_B(x)-\Omega_B(x \pm \kappa)| \sim \kappa^{1/2} \gg N^{-1/2+\epsilon/2} \kappa^{-1/4},
\end{equation*} 
where in the second step we used (\ref{eq_derivativecontrol}) and (\ref{eq_closenessofomega}) with Cauchy's  integral formula when $x>E_+.$ On the other hand, if $\Omega_B^{-1}(\wh{a}_i) \in [x-c\kappa, x+c\kappa]$ such that $\Omega_B^{-1}(\wt{a}_i)-E_+ \sim \kappa,$ here $c<1$ is some small constant.  By (\ref{eq_edgeestimationrough}) and the fact $\wh{a}_i-\Omega_{\beta}(E_{+}) \geq N^{-1/3+\epsilon},$ we have that 
\begin{equation*}
\Omega_B^{-1}(\wh{a}_i)-E_+ \sim \Delta_1(\wh{a}_i)^4  \gg N^{-1/2+\epsilon} \Delta_1(\wh{a}_i).
\end{equation*}
Moreover, by (\ref{eq_derivativecontrol}) and (\ref{eq_closenessofomega}), we conclude that 
\begin{equation*}
|\Omega'_B(\xi)| \sim |\Omega_B'(\Omega_B^{-1}(\wh{a}_i))| \sim \Delta_1(\wh{a}_i)^{-2}, \ \xi \in \mathrm{I}_i^{(a)},
\end{equation*}
where we used (\ref{eq_edgeestimationrough}) in the second step. Since $\Omega_B$ is monotonically increasing on $(E_+, \infty),$ for $x \notin \mathrm{I}_i^{(a)},$ by (\ref{eq_mvtdifference}) and (\ref{eq_edgeestimationrough}), we conclude that 
\begin{align*}
|\Omega_B(x)-\Omega_B(\Omega_B^{-1}(\wh{a}_i))| &\geq |\Omega_B(\Omega_B^{-1}(\wh{a}_i) \pm N^{-1/2+\epsilon} \Delta_1(\wh{a}_i))-\Omega_B(\Omega_B^{-1}(\wh{a}_i))| 
\nonumber \\
& \sim N^{-1/2+\epsilon} \Delta_1(\wh{a}_i)^{-1} \gg N^{-1/2+\epsilon/2} \kappa^{-1/4}.  
\end{align*}
The $d_{\mu}^b$ term can be dealt with in the same way, and this completes our proof. 
\end{proof}

{
\noindent{\bf Step 2:} In this step we will show that each $\mathrm{I}_i^{(a)}$, $i \in  \mathcal{O}_\epsilon^{(a)}$, 
or $\mathrm{I}_\mu^{(b)}$, $\mu\in \mathcal{O}_\epsilon^{(b)}$, contains the right number of eigenvalues of $\widehat{\mathcal{Q}}_1$, under a {\it special case}; see \eqref{eq_multione} below. 
For simplicity, we relabel the indices in $\mathcal{O}_\epsilon^{(a)}\cup \mathcal{O}_\epsilon^{(b)}$ as $\wt\sigma_1, \cdots, \wt\sigma_{r_\epsilon}$, and call them $\epsilon$-{\it{spikes}}. Moreover, we assume that they correspond to classical locations of outliers as $x_1,\cdots, x_{r_\epsilon}$ (some of them are determined by $\Omega_B^{-1}(\widehat{a}_i)$, while others are given by $\Omega_A^{-1}(\widehat{b}_\mu)$), such that  
\begin{equation*}
x_1 \ge x_2 \ge \cdots \ge x_{r_\epsilon}.
\end{equation*}
The corresponding permissible intervals $\mathrm{I}_i^{(a)}$ and $\mathrm{I}_\mu^{(b)}$ are relabelled as $\mathrm{I}_i$, $1\le i \le r_\epsilon$. In this step, we consider a special configuration $\mathbf x\equiv \mathbf x(0):=(x_1, x_2 , \cdots , x_{r_\epsilon})$ of the outliers that is {\it independent of $N$} and satisfies 
\begin{equation}\label{eq_multione}
x_1>x_2 > \cdots > x_{r_\epsilon}>E_+.
\end{equation}
In this step, we claim that each $\mathbf{\mathrm{I}}_i(\mathbf x)$, $1\le  i \le r_\epsilon$, contains precisely one eigenvalue of $\widehat{\mathcal{Q}}_1$. Fix any $1\le i \le r_\epsilon$ and pick up a small $n$-independent positively oriented closed contour $\mathcal{C} \subset \mathbb{C}/[0, E_+]$ that encloses $x_i$ but no other point of the set $\{x_i\}_{i=1}^{r_\epsilon}.$ Define two functions
\begin{equation*}
h(z):=\det(\bm{\mathcal{D}}^{-1}+z \mathbf{U}^* \mathbf{G}(z) \mathbf{U}), \quad  l(z)=\det(\bm{\mathcal{D}}^{-1}+z \mathbf{U}^*\Theta(z) \mathbf{U}).
\end{equation*}  
The functions $h,l$ are holomorphic on and inside $\mathcal{C}$ when $n$ is sufficiently large by \eqref{eq_proofoutlier}. Moreover, by the construction of $\mathcal{C},$ the function $l$ has precisely one zero inside $\mathcal{C}$ at $x_i.$ By \eqref{eq_proofoutlierlocallaw}, we have  
\begin{equation*}
\min_{z \in \mathcal{C}}|l(z)| \gtrsim 1, \quad |h(z)-l(z)|=\mathrm{O}(N^{-1/2+ \epsilon/2}).
\end{equation*}
The claim then follows from Rouch{\' e}'s theorem.

\vspace{3pt}

\noindent{\bf Step 3:} In order to extend the results in Step 2 to arbitrary $N$-dependent configuration $\mathbf x_N$, we shall employ a continuity argument as in \cite[Section 6.5]{KY13}. We first choose an $N$-independent $\mathbf{x}(0)$ that satisfies \eqref{eq_multione}. We then choose a continuous ($N$-dependent) path of the eigenvalues of $D^a$ and $D^b$, which gives a continuous path of the configurations $(\mathbf{x}(t):0\le t \le 1)$ that connects $\mathbf{x}(0)$ and $\mathbf{x}(1)=\mathbf{x}_N$. Correspondingly, we have a continuous path of eigenvalues $\{\widehat \lambda_i(t)\}_{i=1}^N$. We require that $\mathbf{x}(t)$ satisfies the following properties.
\begin{itemize}
\item[(i)] For all $t\in [0,1]$, the eigenvalues of $D^a(t)$ and $D^b(t)$ are all non-negative.
\item[(ii)] For all $t\in [0,1]$, the number $r_\epsilon$ of $\epsilon$-spikes is unchanged and we denote them by $\widehat \sigma_1(t),\cdots, \widehat \sigma_{r_\epsilon}(t)$. Moreover, we always have the following order of the outliers: $x_1(t) \ge x_2(t) \ge \cdots \ge x_{r_\epsilon}(t)$.
\item[(iii)] For all $t\in [0,1]$, we denote the permissible intervals as $\mathrm{I}_i(t)$. If $\mathrm{I}_i(1)\cap \mathrm{I}_j(1)=\emptyset$ for $1\le i < j\le r_\epsilon$, then $\mathrm{I}_i(t)\cap \mathrm{I}_j(t)=\emptyset$ for all $t\in [0,1]$. The interval $\mathrm{I}_0$ in \eqref{I0} is unchanged along the path. 
\end{itemize}
It is easy to see that such a path $\mathbf{x}(t)$ exists. With a bootstrap argument along the path $\mathbf{x}(t)$, we can prove the following lemma and complete Step 3. 
\begin{lem}\label{lem_cont} 
On the event $\Xi$, the estimate \eqref{eq_reduceproofoutlier} holds for the configuration $\mathbf{x}(1)$.
\end{lem}
\begin{proof}
See \cite[Lemma D.3]{DYaos}. 
\end{proof}
}

\vspace{3pt}

\noindent{\bf Step 4:} In this step, we consider the extremal non-outlier eigenvalues when $i \notin \left( \mathcal{O}_{\epsilon}^{(a)} \cup \mathcal{O}_{\epsilon}^{(b)} \right)$ and prove (\ref{eq_reducebulkvalue}). The discussion will use the following eigenvalue interlacing result. 
\begin{lem}\label{lem_interlacing} Recall the eigenvalues of $\wh{\mathcal{Q}}_1$ and $\mathcal{Q}_1$ are denoted as $\{\wh{\lambda}_i\}$ and $\{\lambda_i\},$ respectively. Then we have that
\begin{equation*}
\wh{\lambda}_i \in [\lambda_i, \lambda_{i-r-s}],
\end{equation*}
where we adopt the convention that $\lambda_i=\infty$ if $i<1$ and $\lambda_i=0$ if $i>N.$
\end{lem} 
\begin{proof}
See \cite[Lemma C.3]{DYaos}.
\end{proof}

We first fix a configuration $\mathbf{x}(0)$ as mentioned earlier. Then by the discussion of Step 2, (\ref{eq_proofoutlier}) and Lemma \ref{lem_interlacing}, we can prove (\ref{eq_reducebulkvalue}) under the configuration $\mathbf{x}(0).$  For arbitrary $N$-dependent configuration, we again use a continuity argument as mentioned in Step 3. We refer the readers to Step 4 of the proof of Theorem 3.6 of \cite{DYaos}. This finishes the proof of Theorem \ref{thm_outlier}.

\subsection{Proof of Theorem \ref{thm_eigenvaluesticking}}\label{sec_proofsticking}

In this subsection, we prove the results for the bulk eigenvalues.
According to Theorems \ref{thm_outlier} and \ref{prop_linearlocallaw} and Lemma \ref{thm_rigidity}, for any fixed $\epsilon>0,$ we can choose a high probability $\Xi_1$ in which (\ref{eq_proofinsidelocallaw})--(\ref{eq_proofoutlier}) and the following estimates hold;
\begin{equation}\label{eq_stickingone}
\mathbf{1}(\Xi_1)|\widehat{\lambda}_i-E_+| \leq N^{-2/3+\epsilon/2}, \ \text{for} \ r_++s_++1 \leq i \leq \varpi,
\end{equation}
for some integer $\varpi \geq r+s;$ and for $i \leq \tau N,$
\begin{equation}\label{eq_stickingtwo}
\mathbf{1}(\Xi_1)|\lambda_i-\gamma_i^*| \leq i^{-1/3} N^{-2/3+\epsilon/2}. 
\end{equation}
In what follows, we focus on our discussion on $\Xi_1$ and hence it will be purely deterministic. We follow the three steps discussed in Section \ref{sec_proofstartegy} to conclude Theorem \ref{thm_eigenvaluesticking}. The proof closely follows that of \cite[Theorem 3.7]{DYaos}, and we only focus on pointing out the differences due to similarity. In the first step, we find the permissible region for the eigenvalues and record the results in Lemma \ref{lem_nonoutlierpermissible} below.  For any $i$ and { and $\gamma$ defined in (\ref{eq_defngamma})}, we define the set
\begin{equation*}
\Omega_i\equiv\Omega_{i}(c_{0}):=\left\{ x \in [\lambda_{i-r-s-1}, E_++c_0 N^{-2/3+2 \epsilon}]: \mathrm{dist}(x, \mathrm{spec}(\mathcal{Q}_1))>N^{-1+\epsilon} \gamma^{-1} \right\},
\end{equation*}
where $\mathrm{spec}(\mathcal{Q}_1)$ stands for the spectrum of $\mathcal{Q}_1$.
\begin{lem}\label{lem_nonoutlierpermissible} For $\gamma \geq N^{-1/3+\epsilon}$ and $i \leq N^{1-2 \epsilon} \gamma^3,$ there exits a constant $c_0>0$ so that the set $\Omega_i$ contains no eigenvalue of $\widehat{\mathcal{Q}}_1.$
\end{lem}
\begin{proof}
The proof is similar to that of Lemma D.4 of \cite{DYaos} and we only sketch the key points here. Denote 
\begin{equation*}
\eta_x:=N^{-1+\epsilon} \gamma^{-1}, \ z_x=x+\ii \eta_x. 
\end{equation*}
 Using a discussion similar to (\ref{eq_sufficientconverse}), we find that $x$ is not an eigenvalues of $\widehat{\mathcal{Q}}_1$ if 
 \begin{equation}\label{eq_newcondition}
 \min\left\{ \min_{1 \leq i \leq r}\left| \frac{d_i^a+1}{d_i^a}+\frac{\Omega_B(x)}{a_i-\Omega_B(x)} \right|, \min_{1 \leq \mu-N \leq s } \left| \frac{d_\mu^b+1}{d_\mu^b}+\frac{\Omega_A(x)}{b_\mu-\Omega_A(x)} \right| \right\} \gg N^{\epsilon/2} \im \Omega_B(z_x)+\frac{N^{\epsilon/2}}{N \eta_x}.
 \end{equation}
On one hand, using the definition of $\gamma,$ it is easy to see that for $x \in \Omega_i$  and some constant $C>0$
\begin{equation*}
 \min\left\{ \min_{1 \leq i \leq r}\left| \frac{d_i^a+1}{d_i^a}+\frac{\Omega_B(x)}{a_i-\Omega_B(x)} \right|, \min_{1 \leq \mu-N \leq s } \left| \frac{d_\mu^b+1}{d_\mu^b}+\frac{\Omega_A(x)}{b_\mu-\Omega_A(x)} \right| \right\} \geq C \gamma.
\end{equation*}
On the other hand, using \eqref{eq_localquadtratic} and Lemma \ref{prop:stabN}, we readily obtain that for $x \in \Omega_i$
\begin{equation*}
N^{\epsilon/2} \im \Omega_B(z_x)+\frac{N^{\epsilon/2}}{N \eta_x} \ll \gamma. 
\end{equation*}
This completes the proof using (\ref{eq_newcondition}). 
\end{proof}
In the second step, we perform the counting argument for a special case as in Lemma \ref{lem_nonoutcounting} below. 
\begin{lem}\label{lem_nonoutcounting} We fix a configuration $\mathbf{x} \equiv \mathbf{x}(0):=(x_1, x_2, \cdots, x_{r+s})$ of the outliers that is independent of $N$ and satisfies 
\begin{equation*}
x_1>x_2>\cdots>x_{r+s}>E_+.
\end{equation*} 
For $\gamma \geq N^{-1/3+2 \epsilon}$ and $i \leq N^{1-4 \epsilon} \gamma^3,$ we have that for some constant $C>0$ 
\begin{equation*}
|\widehat{\lambda}_{i+r+s}-\lambda_i| \leq C N^{-1+2 \epsilon} \gamma^{-1}. 
\end{equation*}
\end{lem}
\begin{proof}
Since the proof is similar to the counterpart Lemma D.5 of \cite{DYaos}, we only sketch it. The key idea is to group together the eigenvalues that are close to each other. Let $\mathcal{A}=\{A_k\}$ be the finest partition of $\{1,2,\cdots, N\}$ such that $i<j$ belong to the same block of $\mathcal{A}$ if 
\begin{equation*}
|\lambda_i-\lambda_j| \leq \delta:=N^{-1+7 \epsilon/6} \gamma^{-1}. 
\end{equation*}
Note that each block $A_k$ of $\mathcal{A}$ consists of a sequence of consecutive integers. Denote $k^*$ such that $N^{1-4 \epsilon} \gamma^3 \in A_{k^*}.$ Following a discussion similar to (D.47) and (D.48) of \cite{DYaos}, we can conclude that 
\begin{equation}\label{eq_sizeeigen}
|A_k| \leq CN^{3\epsilon/4}  \quad \text{for } \ k=1,\cdots, k^*,
\end{equation}
and for any given $i_k \in A_k$,
\begin{equation}\label{rigidity_size}
|\lambda_{i}-\gamma^*_{i_k}|\leq i^{-1/3} N^{-2/3+ \epsilon}  \quad \text{for all }\ i \in A_k .
\end{equation}
For any $1 \leq k \leq k^*,$ we denote 
\begin{equation*}
a^k:=\min_{i \in A_k} \lambda_{i} = \lambda_{m_k}, \quad b^k:=\max_{i \in A_k} \lambda_{i}=\lambda_{l_k}. 
\end{equation*}
With the above notations, we introduce the continuous path as
\begin{equation*}
x_t^k=(1-t)\left(a^k - \delta/3\right)+t\left(b^k  + \delta/3\right)t, \quad t \in [0,1].
\end{equation*}
Note that $x_0^k= a^k - \delta/3$ and $x_1^k= b^k  + \delta/3$. The interval $[x_0^k, x_1^k]$ contains precisely the eigenvalues of $\mathcal Q_1$ that are in $A_k$, and the endpoint $x_0^k$ (or $x_1^k$) is at a distance at least of the orders $\delta/3$  from any eigenvalue of $\mathcal Q_1$. 

Moreover, on one hand, using a standard perturbation approach as in Proposition D.6 of \cite{DYaos}, we can show that $\widehat{\mathcal Q}_1$ has at least $|A_k|$ eigenvalues in $[x_0^k, x_1^k]$ for $1\le k \le k^*$. On the other hand, by (\ref{eq_sizeeigen}), (\ref{rigidity_size}) and Lemma \ref{lem_interlacing}, we conclude that $\widehat{\mathcal Q}_1$ has at most $|A_k|$ eigenvalues in $[x_0^k, x_1^k].$ This concludes the proof.  
\end{proof}
In the third step, we generalize Lemma \ref{lem_nonoutcounting} using a continuity argument as in the proof of Theorem \ref{thm_outlier}. We omit the details and conclude the proof of Theorem \ref{thm_eigenvaluesticking}.

\section{Proofs of Theorems \ref{thm_outlieereigenvector} and \ref{thm_nonoutliereigenvector}: the eigenvector statistics}\label{sec_proofeigenvectors}

In this section, we prove the main results of Section \ref{sec_eigenvectorresults} following the proof strategy outlined in Section \ref{sec_proofstartegy}.  We again focus on explaining the main differences from the counterparts in \cite{Bloemendal2016, DYaos}  and how to adapt their proof strategies. As mentioned before, the proof of Theorem \ref{thm_outlieereigenvector} contains two parts. In Section \ref{sec_prove36withconstrain} we prove the results under the assumption of (\ref{condition_nonoverone}) and then remove this assumption to complete the proof in Section \ref{sec_prove36withoutconstrain}. 

\subsection{Proof of Theorem \ref{thm_outlieereigenvector} under (\ref{condition_nonoverone})}\label{sec_prove36withconstrain}

In this subsection, we prove Theorem \ref{thm_outlieereigenvector} assuming (\ref{condition_nonoverone}). Due to similarity, we only focus on the left singular vectors.  The main technical task of this section is to prove Proposition \ref{prop_maineigenvector}, which implies the results. {Recall the definitions in (\ref{eq_shorthandnotationsdelta}) and (\ref{eq_differencedefinition}). }
\begin{prop}\label{prop_maineigenvector} Suppose the assumptions of Theorem \ref{thm_outlieereigenvector} and (\ref{condition_nonoverone}) hold. Then for all $i,j=1,2,\cdots, N,$ we have that
\begin{align*}
& \left|\langle \mathbf{e}_i, \mathcal{P}_S \mathbf{e}_j \rangle-\delta_{ij} {\mathbf{1}}(\pi_a(i) \in S) \wh{a}_i \frac{(\Omega_B^{-1})'(\wh{a}_i)}{\Omega_B^{-1}(\wh{a}_i)} \right|\prec \frac{{\mathbf{1}}(\pi_a(i) \in S, \pi_a(j) \in S)}{\sqrt{N} \sqrt{\Delta_1(\wh{a}_i) \Delta_1(\wh{a}_j)}}+\frac{{\mathbf{1}}(\pi_a(i) \in S, \pi_b(j) \notin S)\Delta_1(\wh{a}_i)}{\sqrt{N} \delta^a_{\pi_a(i), \pi_b(j)}} \\
&+ \frac{1}{N} \left( \frac{1}{\delta_{\pi_a(i)}(S)}+\frac{{\mathbf{1}}(\pi_a(i) \in S)}{\Delta_1(\wh{a}_i)^2} \right) \left( \frac{1}{\delta_{\pi_a(j)}(S)}+\frac{{\mathbf{1}}(\pi_a(j) \in S)}{\Delta_1(\wh{a}_j)^2} \right)+(i \leftrightarrow j),
\end{align*}
where $(i \leftrightarrow j)$ denotes the same terms but with $i$ and $j$ interchanged. 
\end{prop}

\begin{proof}[\bf Proof of Theorem \ref{thm_outlieereigenvector} under (\ref{condition_nonoverone})] For the left singular vectors,  using that $\caO^{+}$ is finite and
\begin{equation}\label{eq_eigenvectorexpansion}
\mathbf{v}=\sum_{k=1}^N \langle \mathbf{e}_k, \mathbf{v} \rangle \mathbf{e}_k=\sum_{k=1}^N v_k \mathbf{e}_k,
\end{equation}
the results simply follow from Proposition \ref{prop_maineigenvector}. The proof for the right singular vectors is analogous.
\end{proof}

The rest of the subsection is devoted to the proof of Proposition \ref{prop_maineigenvector}. Its proof consists of two steps. In the first step, we prove the results under the following \emph{non-overlapping} condition, which guarantees a phenomenon of cone concentration. 
 \begin{assu}\label{assum_eigenvector} For some fixed small constant $\tau_2>0,$ we assume that for all $\pi_a(i) \in S$ and $\pi_b(\mu) \in S$
 \begin{equation}\label{conidition_nonovertwo}
 \delta_{{\pi_a(i)}}(S) \geq N^{-1/2+\tau_2} {\Delta_1^{-1}(\widehat{a}_i)},  \ \delta_{\pi_b(\mu)}(S) \geq N^{-1/2+\tau_2} {\Delta_2^{-1}(\widehat{b}_\mu)}. 
 \end{equation}
 \end{assu}

 Let $\omega<\tau_1/2$ and $0<\epsilon<\min\{\tau_1,\tau_2\}/10$ be some small positive constants to be chosen later. By Theorems \ref{prop_linearlocallaw} and \ref{thm_outlier} and Lemma \ref{thm_rigidity}, we can choose a high probability event $\Xi_2 \equiv \Xi_2(\epsilon,\omega, \tau_1,\tau_2)$ where the following statements hold. 
\begin{enumerate}
\item[(i)] For all {
\begin{equation}\label{eq_evoutsideparameter}
z \in \mathcal{T}_{out}(\omega):=\left\{E+\ii \eta \in \mathbb{C}: E_++N^{-2/3+\omega} \leq E \leq \omega^{-1} \right\},
\end{equation} }
we have that 
\begin{equation}\label{eq_xi1setestimation}
{\mathbf{1}}(\Xi_2)\| \Ub^* (\Gb(z)-\Theta(z)) \Ub \| \leq N^{-1/2+\epsilon} (\kappa+\eta)^{-1/4},
\end{equation}
{where we recall the definition of $\Theta(z)$ in (\ref{eq_Bigthetadefinition}). }
\item[(ii)] Recall the notations in  (\ref{eq_shorthandnotationsdelta}).  For all $1 \leq i \leq r^+$ and $1 \leq \mu-N \leq s^+$, we have 
\begin{equation}\label{eq_outlierlocationrate}
{\mathbf{1}}(\Xi_2) \left| \widehat{\lambda}_{\pi_a(i)}-\Omega_B^{-1}(\wh{a}_i) \right| \leq N^{-1/2+\epsilon} {\Delta_1(\widehat{a}_i)},\ {\mathbf{1}}(\Xi_2) \left| \wh{\lambda}_{\pi_b(\mu)}-\Omega_A^{-1}(\wh{b}_\mu) \right| \leq N^{-1/2+\epsilon}{\Delta_2(\widehat{b}_\mu)}.
\end{equation}
\item[(iii)] For any fixed integer $\varpi>r+s$ and all $r^++s^+<i \leq \varpi,$ we have that 
\begin{equation}\label{eq_edgerigidity}
{\mathbf{1}}(\Xi_2) \left[ |\lambda_1-E_+|+|\wh{\lambda}_i-E_+| \right] \leq N^{-2/3+\epsilon}. 
\end{equation}  
\end{enumerate}
From now on, we will focus our discussion on the high probability event $\Xi_2$ and hence all the discussion will be purely deterministic. 

{ Here we briefly pause to discuss consequences of the assumption \eqref{condition_nonoverone}. First of all, note that for any fixed $\epsilon>0$
\beq\label{eq_condition_conseq}
\Omega_{B}(\wh{E}_{+})\leq \Omega_{B}(E_{+}+N^{-2/3+\epsilon})\leq \Omega_{\beta}(E_{+}+N^{-2/3+\epsilon})+ N^{-2/3+\epsilon/2}\leq \Omega_{\beta}(E_{+})+2N^{-1/3+\epsilon/2},
\eeq
where we used Lemma \ref{lem_prop_fc} (v) and (vi) in the first inequality, Lemma \ref{prop:stabN} (4) in the second inequality, and Lemma \ref{lem_prop_fc} (iii) in the last inequality. Then, recalling \eqref{condition_nonoverone} and $S\subset\caO_{+}$, for any $i$ with $\pi_{a}(i)\in S$ we have 
\beqs
\wh{a}_{i}\geq \Omega_{\beta}(E_{+})+N^{-1/3+\tau_{1}}\geq \Omega_{B}(\wh{E}_{+})+N^{-1/3+\tau_{1}/2},
\eeqs
where we took $\epsilon<\tau_{1}$ in \eqref{eq_condition_conseq}. Hence $\wh{a}_{i}$ is contained in the analytic domain of $\Omega_{B}^{-1}$ and in particular $\Omega_{B}^{-1}(\wh{a}_{i})$ is well-defined. By the exact same reasoning we have $\wh{b}_{\mu}>\Omega_{A}(\wh{E}_{+})$ and $\Omega_A^{-1}(\wh{b}_{\mu})$ is well-defined. }

We next define {a contour} \normalcolor that will be used in the rest of the proof. Recall (\ref{eq_alphasinside}) and (\ref{eq_alphaoutside}). Denote 
\begin{equation*}
\rho_i^a=c_i\left[ \delta_{\pi_a(i)}(S) \wedge {\Delta_1^2(\widehat{a}_i)} \right], \ \pi_a(i) \in S,
\end{equation*}
and 
\begin{equation*}
\rho_\mu^b=c_\mu \left[ \delta_{\pi_b(\mu)}(S) \wedge {\Delta_2^2(\widehat{b}_\mu)} \right], \ \pi_b(\mu) \in S,
\end{equation*}
for some sufficiently small constants $0<c_i, c_\mu<1.$ Define the contour $\Gamma:=\partial \mathsf{C}$ as the boundary of the union of the open discs 
\begin{equation}\label{eq_defndiscs}
\mathsf{C}:=\bigcup_{\pi_a(i) \in S} B_{\rho_i^a}(\wh{a}_i) \ \cup \bigcup_{\pi_b(\mu) \in S} B_{\rho_\mu^b}(\Omega_B(\Omega_A^{-1}(\wh{b}_\mu))),
\end{equation}
where $B_\mathsf{r}(x)$ denotes an open disc of radius $\mathsf{r}$ around $x.$ It is easy to see that the contour $\mathsf{C}$ is in the analytic domain of $\Omega_B^{-1}.$ In the following lemma, we will show that by choosing sufficiently small $c_i, c_\mu,$ we have that (1) $\overline{\Omega_B^{-1}(\mathsf{C})}$ is a subset of (\ref{eq_evoutsideparameter}) and hence (\ref{eq_xi1setestimation}) holds; (2) $\partial \Omega_B^{-1}(\mathsf{C})=\Omega_B^{-1}(\Gamma)$ only encloses the outliers with indices in $S.$ Its proof is similar to \cite[Lemmas 5.4 and 5.5]{Bloemendal2016} or \cite[Lemma E.6]{DYaos} utilizing the results in Section \ref{sec_subordinationpre}. We omit the details here.

\begin{lem}\label{lem_contour} Suppose that the assumptions of Theorem \ref{thm_outlieereigenvector} and (\ref{condition_nonoverone}) hold true. Then the set $\overline{\Omega_B^{-1}(\mathsf{C})}$ lies in the parameter set (\ref{eq_evoutsideparameter}) as long as $c_i$'s and $c_\mu$'s are sufficiently small.  Moreover, we have that $\{\wh{\lambda}_{\mathfrak{a}}\}_{\mathfrak{a} \in S} \subset \Omega_B^{-1}(\mathsf{C})$ and all the other eigenvalues lie in the complement of $\overline{\Omega_B^{-1}(\mathsf{C})}.$ 
\end{lem}

Then we prove Proposition \ref{prop_maineigenvector}. The proof follows from the {same} strategy as \cite[Proposition E.5]{DYaos}.
\begin{proof}[\bf Proof of Proposition \ref{prop_maineigenvector}] The proof consists of two steps. In the first step, we prove the results under Assumption \ref{assum_eigenvector}. In the second step, we { prove the results by removing Assumption \ref{assum_eigenvector} from the first step}. Since the second step is rather standard, we focus on the first step and only briefly discuss the second step.
 
\vspace{2pt}
\noindent{\bf Step 1:} In the first step, we prove the results assuming that the non-overlapping condition Assumption \ref{assum_eigenvector} holds. For convenience, we set $d_i^a=0$ when $i>r$ and deal with the general case at the end of this step. Denote 
\begin{equation}\label{eq_defnez}
\mathcal{E}(z)=z\Ub^*(\Theta(z)-\Gb(z))\Ub. 
\end{equation}
Using the resolvent expansion, we obtain that 
\begin{align}\label{eq_resolventexpansion}
\frac{1}{\bm{\mathcal{D}}^{-1}+z \Ub^* \Gb(z) \Ub}&=\frac{1}{\bm{\mathcal{D}}^{-1}+z \Ub^* \Theta(z) \Ub}+\frac{1}{\bm{\mathcal{D}}^{-1}+z \Ub^* \Theta(z) \Ub} \mathcal{E} \frac{1}{\bm{\mathcal{D}}^{-1}+z \Ub^* \Theta(z) \Ub} \nonumber \\
&+\frac{1}{\bm{\mathcal{D}}^{-1} +z \Ub^* \Theta(z) \Ub} \mathcal{E} \frac{1}{\bm{\mathcal{D}}^{-1}+z \Ub^* \Gb(z) \Ub} \mathcal{E} \frac{1}{\bm{\mathcal{D}}^{-1}+z \Ub^* \Theta(z) \Ub}.
\end{align} 
Together with (\ref{eq_intergralrepresentation}) and (\ref{eq_keyexpansionev}), using the fact that $\Gamma$ does not enclose any pole of $\Gb$ (due to (\ref{eq_edgerigidity})), we have the following decomposition 
\begin{equation*}
\langle \mathbf{e}_i, \mathcal{P}_S \mathbf{e}_j \rangle =\frac{\sqrt{(1+d_i^a)(1+d_j^a)}}{d_i^a d_j^a}(s_0+s_1+s_2),
\end{equation*}
where $s_0, s_1$ and $s_2$ are defined as 
\begin{equation*}
s_0=\frac{\delta_{ij}}{2 \pi \ii} \oint_{\Omega_B^{-1}(\Gamma)} \frac{1}{(d_i^a)^{-1}+a_i(a_i-\Omega_B(z))^{-1}} \frac{\dd z}{z},
\end{equation*}
\begin{equation*}
s_1=\frac{1}{2 \pi \ii} \oint_{\Omega_B^{-1}(\Gamma)} \frac{\mathcal{E}_{ij}(z)}{((d_i^a)^{-1}+a_i(a_i-\Omega_B(z))^{-1})((d_j^a)^{-1}+a_j(a_j-\Omega_B(z))^{-1})} \frac{\dd z}{z},
\end{equation*}
\begin{equation*}
s_2=\frac{1}{2 \pi \ii} \oint_{\Omega_B^{-1}(\Gamma)} \left(\frac{1}{\bm{\mathcal{D}}^{-1} +z \Ub^* \Theta(z) \Ub} \mathcal{E} \frac{1}{\bm{\mathcal{D}}^{-1}+z \Ub^* \Gb(z) \Ub} \mathcal{E} \frac{1}{\bm{\mathcal{D}}^{-1}+z \Ub^* \Theta(z) \Ub} \right)_{ij} \frac{\dd z}{z}.
\end{equation*}

First, we we deal with the term containing $s_0.$ Using Cauchy's integral formula, we readily see that 
\begin{align*}
\frac{\sqrt{(1+d_i^a)(1+d_j^a)}}{d_i^a d_j^a}s_0 &=\frac{\sqrt{(1+d_i^a)(1+d_j^a)}}{d_j^a}  \frac{\delta_{ij}}{2 \pi \ii} \oint_{\Gamma} \frac{(\Omega_B^{-1})'(\zeta)}{\Omega_B^{-1}(\zeta)} \frac{a_i-\zeta}{\wh{a}_i-\zeta} \dd \zeta \\
&= \delta_{ij} \wh{a}_i \frac{(\Omega_B^{-1})'(\wh{a}_i)}{\Omega_B^{-1}(\wh{a}_i)}.
\end{align*}

Second, we control the term containing $s_1.$ We use the change of variables $\zeta=\Omega_{B}(z)$ to obtain that 
\begin{equation}\label{eq_s1decomposition}
s_1=\frac{d_i^a d_j^a}{2 \pi \ii} \oint_{\Gamma} \frac{\xi_{ij}(\zeta)}{(\zeta-\wh{a}_i)(\zeta-\wh{a}_j)} \dd \zeta, \ \xi_{ij}(\zeta)=(\zeta-a_i)(\zeta-a_j) \mathcal{E}_{ij}(\Omega_B^{-1}(\zeta)) \frac{(\Omega_B^{-1})'(\zeta)}{\Omega_B^{-1}(\zeta)}.
\end{equation}
To bound $\xi_{ij}(\zeta),$ we first prepare some useful estimates. When $c_i$'s and $c_\mu$'s are sufficiently small, by a discussion similar to (\ref{eq_edgeestimationrough}), we have that for $\zeta \in \Gamma,$
\begin{equation}\label{eq_complexedgederterministicbound}
|\Omega_{B}^{-1}(\zeta)-E_+| \sim |\zeta-\Omega_{\beta}(E_{+})|^2.
\end{equation} 
Moreover, for any $z_b\in\Omega_B^{-1}(\Gamma)$, by Cauchy's differentiation formula we obtain that 
\begin{equation}\label{eq_cauchydifferentation}
\Omega_B'(z_b)-\Omega_{\beta}'(z_b)=\frac{1}{2 \pi \ii} \oint_{\mathcal{C}_b} \frac{\Omega_B(\xi)-\Omega_{\beta}(\xi)}{(\xi-z_b)^2} \dd \xi,
\end{equation}   
where $\mathcal{C}_b$ is the disc of radius $|z_b-E_+|/2$ centered at $z_b.$ Here we used the facts that both $\Omega_B$ and $\Omega_{\beta}$ are holomorphic on $\Omega_B^{-1}(\Gamma).$ Together with  (\ref{eq_derivativecontrol}), using the residue theorem, we readily see that for some constant $C>0$
\begin{equation*}
\Omega_B'(z_b) \sim C N^{-1/2+\epsilon} |z_b-E_+|^{-1}+|z_b-E_+|^{-1/2} \sim |z_b-E_+|^{-1/2},  
\end{equation*} 
where in the last step we used (\ref{eq_complexedgederterministicbound}) such that $|z_b-E_+| \geq CN^{-2/3+\epsilon}.$ Consequently, using implicit differentiation, we conclude that 
\begin{equation}\label{eq_derivativeinversebound}
(\Omega_B^{-1})'(\zeta) \sim |z_b-E_+|^{1/2} \sim |\zeta-\Omega_{\beta}(E_{+})|, 
\end{equation}  
where in the last step we used (\ref{eq_complexedgederterministicbound}). With the above preparation, we now proceed to control $\xi_{ij}(\zeta)$ and $s_1.$ By (\ref{eq_xi1setestimation}), (\ref{eq_complexedgederterministicbound}) and (\ref{eq_derivativeinversebound}), it is easy to see that for $\zeta \in \Gamma$
\begin{equation}\label{eq_xiijwholebound}
|\xi_{ij}(\zeta)| \lesssim N^{-1/2+\epsilon} |\zeta-\Omega_{\beta}(E_{+})|^{1/2}. 
\end{equation}
Together with a discussion similar to (\ref{eq_cauchydifferentation}), we see that  
\begin{equation}\label{eq_xiijderivativewholebound}
|\xi_{ij}'(\zeta)| \lesssim N^{-1/2+\epsilon} |\zeta-\Omega_{\beta}(E_{+})|^{-1/2}. 
\end{equation}
In order to control $s_1,$ we will consider different cases. In the first case when both $\pi_a(i) \in S$ and $\pi_b(j) \in S,$ if $\wh{a}_i \neq \wh{a}_j,$ by (\ref{eq_s1decomposition}) and the residue theorem, we have that 
\begin{align*}
|s_1| \leq C \left| \frac{\xi_{ij}(\wh{a}_i)-\xi_{ij}(\wh{a}_j)}{\wh{a}_i-\wh{a}_j} \right| \leq \frac{C}{|\wh{a}_i-\wh{a}_j|} \left| \int^{\wh{a}_i}_{\wh{a}_j}|\xi'_{ij}(\zeta)| \dd \zeta \right| \leq \frac{CN^{-1/2+\epsilon}}{\sqrt{{\Delta_1(\wh{a}_i) \Delta_1(\wh{a}_j)}}},
\end{align*} 
where we used (\ref{eq_xiijderivativewholebound}) in the last step. The same argument applies to the case $\wh{a}_i=\wh{a}_j$. In the second case when $\pi_a(i) \in S$ and $\pi_a(j) \notin S,$ we conclude from (\ref{eq_xiijwholebound}) that 
\begin{equation*}
|s_1| \leq C  \frac{\left| \xi_{ij}(\wh{a}_i) \right|}{\left| \wh{a}_i-\wh{a}_j \right|} \leq \frac{C \Delta_1(\wh{a}_i) N^{-1/2+\epsilon}}{\delta^a_{\pi_a(i), \pi_a(j)}}. 
\end{equation*}
Similarly, we can estimate $s_1$ when $\pi_a(i) \notin S$ and $\pi_b(j) \in S.$ Finally, when both $\pi_a(i) \notin S$ and $\pi_b(j) \notin S,$ we have $s_1=0$ by the residue theorem. This completes the estimation regarding $s_1.$ 

We then estimate $s_2,$ which relies on some crucial estimates on the contour. We decompose $\Gamma$ as
\begin{equation}\label{decompose_contour}
\Gamma=\bigcup_{\pi_a(i) \in S} \Gamma_i \cup \bigcup_{\pi_b(\mu) \in S} \Gamma_\mu, \ \Gamma_i:=\Gamma \cap \partial B_{\rho_i^a}(\wh{a}_i), \ \Gamma_\mu:=\Gamma \cap \partial B_{\rho_\mu^b}(\Omega_B(\Omega_A^{-1}(\wh{b}_\mu))).
\end{equation}    
The following lemma is the key input for $s_2.$ Its proof is similar to Lemma E.7 of \cite{DYaos} and omitted here. 
\begin{lem}\label{lem_contourestimation} For any $\pi_a(i) \in S, 1 \leq j \leq r, 1 \leq \nu-N \leq s$ and $\zeta \in \partial B_{\rho_i^a}(\wh{a}_i),$ we have that 
\begin{equation}\label{eq_onlyproveresult}
\left| \zeta-\wh{a}_j \right| \sim \rho_i^a+\delta^a_{\pi_a(i), \pi_a(j)}, 
\end{equation}
and 
\begin{equation*}
\left| \Omega_A(\Omega_B^{-1}(\zeta))-\wh{b}_\nu \right| \sim \rho_i^a+\delta^a_{\pi_a(i), \pi_b(\nu)}.
\end{equation*}
For any $\pi_b(\mu) \in S, 1 \leq j \leq r, 1 \leq \mu-N \leq s$ and $\zeta \in \partial B_{\rho_\mu^b}(\Omega_B(\Omega_A^{-1}(\wh{b}_\mu))),$ we have 
\begin{equation*}
\left| \zeta-\wh{a}_i \right| \sim \rho_\mu^b+\delta^b_{\pi_b(\mu), \pi_a(j)},
\end{equation*}
and 
\begin{equation*}
\left| \Omega_A(\Omega_B^{-1}(\zeta))-\wh{b}_\nu  \right| \sim \rho_\mu^b+\delta^b_{\pi_b(\mu), \pi_b(\nu)}.
\end{equation*}
\end{lem} 
Then $s_2$ is estimated as follows, using (\ref{eq_xi1setestimation}), (\ref{eq_xiijwholebound}), (\ref{eq_complexedgederterministicbound}), and (i) of Proposition \ref{prop:stabN}.
\begin{align}\label{eq_s3bound}
|s_2| & \leq C \oint_{\Gamma} \frac{N^{-1/2+\epsilon}}{|\zeta-\wh{a}_i||\zeta-\wh{a}_j|} \frac{(\Omega_B^{-1})'(\zeta)}{|\zeta-\Omega_{\beta}(E_{+})|} \left \| \bm{\mathcal{D}}^{-1}+ \Omega_B^{-1}(\zeta) \Ub^* \Gb(\Omega_B^{-1}(\zeta)) \Ub \right\| |\dd \zeta|, \nonumber \\
& \leq C \oint_{\Gamma}\frac{N^{-1+2\epsilon}}{|\zeta-\wh{a}_i| |\zeta-\wh{a}_j|} \frac{1}{\mathfrak{d}(\zeta)-\|\mathcal{E}(\Omega_B^{-1}(\zeta)) \|} |\dd \zeta|,
\end{align}
where $\mathfrak{d}(\zeta)$ is defined as 
\begin{equation*}
\mathfrak{d}(\zeta):=\left( \min_{1 \leq j \leq r}\left| \wh{a}_j-\Omega_B(\zeta) \right| \right) \wedge \left( \min_{1 \leq \mu-N \leq s}|\wh{b}_{\mu}-\Omega_A(\Omega_B^{-1}(\zeta))| \right).
\end{equation*}
Moreover, by (\ref{eq_xi1setestimation}) and (\ref{eq_complexedgederterministicbound}), for some constant $C>0,$ we can bound 
\begin{equation}\label{eq_smallbound}
\| \mathcal{E}(\Omega_B^{-1}(\zeta)) \| \leq C \sqrt{rs} N^{-1/2+\epsilon}|\zeta-\Omega_{\beta}(E_{+})|^{-1/2}.
\end{equation}
Recall that both $r$ and $s$ are bounded. Together with Lemma \ref{lem_contourestimation} and the fact $\epsilon<\tau_2$, we obtain that
\begin{equation*}
\| \mathcal{E}(\Omega_B^{-1}(\zeta)) \| \ll {\Delta_1^{-1}(\widehat{a}_i)} N^{-1/2+\tau_2} \lesssim
\begin{cases}
\rho_i^a \lesssim \mathfrak{d}(\zeta), & \text{for} \ \zeta \in \Gamma_i \\
\rho_\mu^b \lesssim \mathfrak{d}(\zeta), & \text{for} \ \zeta \in \Gamma_\mu
\end{cases},
\end{equation*}
where we used (\ref{eq_smallbound}) and Assumption \ref{assum_eigenvector}. Based on the above estimates, we arrive at 
\begin{equation}\label{eq_boundsim1}
\frac{1}{\mathfrak{d}(\zeta)-\|\mathcal{E}(\Omega_B^{-1}(\zeta)) \|} \lesssim
\begin{cases}
(\rho_i^a)^{-1}, & \text{for} \ \zeta \in \Gamma_i \\
(\rho_\mu^b)^{-1}, & \text{for} \ \zeta \in \Gamma_\mu
\end{cases}.
\end{equation}
Now we proceed to control $s_2.$ Decomposing the integral contour in \eqref{eq_s3bound} as in \eqref{decompose_contour}, using \eqref{eq_boundsim1} and Lemma \ref{lem_contourestimation}, and recalling that the length of $\Gamma_i$ (or $\Gamma_\mu$) is at most $2 \pi \rho_i^a$ (or $2\pi \rho_\mu^b$), we get that for some constant $C>0,$
\begin{equation}\label{estimate_s21}
\begin{split}
|s_2| \le  C \sum_{\pi_a(k) \in S} \frac{N^{-1+2\epsilon}}{(\rho_k^a + \delta^a_{\pi_a(k),\pi_a(i)})(\rho_k^a + \delta^a_{\pi_a(k),\pi_a(j)}) } +C \sum_{\pi_b(\mu) \in S} \frac{N^{-1+2\epsilon}}{(\rho_\mu^b + \delta^b_{\pi_b(\mu),\pi_a(i)})(\rho_\mu^b + \delta^b_{\pi_b(\mu),\pi_a(j)})} .
\end{split}
\end{equation}
Now we bound the right-hand side of \eqref{estimate_s21} using Cauchy-Schwarz inequality. For $\pi_a(i)\notin S$, we have
\begin{align*}
\sum_{\pi_a(k)\in S}\frac{1}{(\rho_k^a + \delta^a_{\pi_a(k),\pi_a(i)})^2 }+\sum_{\pi_b(\mu)\in S}\frac{1}{(\rho_\mu^b + \delta^b_{\pi_b(\mu),\pi_a(i)})^2 }  & \le \sum_{\pi_a(k)\in S}\frac{1}{(\delta^a_{\pi_a(k),\pi_a(i)})^2 }+\sum_{\pi_b(\mu)\in S}\frac{1}{ (\delta^b_{\pi_b(\mu),\pi_a(i)})^2 } \\
& \le \frac{C}{\delta_{\pi_a(i)}(S)^2} .
\end{align*}
For $\pi_a(i)\in S$, we have $\rho_k^a + \delta^a_{\pi_a(k),\pi_a(i)}\gtrsim\rho_i^a $ for $\pi_a(k)\in S$, and $\rho_\mu^b + \delta^b_{\beta(\mu),\pi_a(i)}\gtrsim \rho_i^a$ for $\pi_b(\mu)\in S$. Then we have for some constant $C>0$
$$\sum_{\pi_a(k)\in S}\frac{1}{(\rho_k^a + \delta^a_{\pi_a(k),\pi_a(i)})^2 }+\sum_{\pi_b(\mu)\in S}\frac{1}{(\rho_\mu^b + \delta^b_{\pi_b(\mu),\pi_a(i)})^2 }  \le  \frac{C}{(\rho_i^a)^2}\le \frac{C}{\delta_{\pi_a(i)}(S)^2} + \frac{C}{\Delta_1(\wh{a}_i)^4} .$$
Plugging the above two estimates into \eqref{estimate_s21}, we get that
\begin{equation*}
\begin{split}
|s_2| \le  CN^{-1+2\epsilon}\left(\frac1{\delta_{\pi_a(i)}(S)} + \frac{{\mathbf{1}}(\pi_a(i)\in S)}{\Delta_1(\wh{a})^2} \right)\left(\frac1{\delta_{\pi_a(j)}(S)} + \frac{{\mathbf{1}}(\pi_a(j)\in S)}{\Delta_1(\wh{a}_j)^2} \right).
\end{split}
\end{equation*}
So far, we have proved Proposition \ref{prop_maineigenvector} for $1\le i,j \le r$ since $\epsilon$ can be arbitrarily small.

Finally, the general case can be dealt with easily. For general $i,j \in \{1,\cdots, N\}$, we define $\mathcal R:=\{1,\cdots, r\}\cup \{i,j\}$. Then we define a perturbed model as 
\begin{equation}\nonumber
\wh{\mathcal{A}} = A\Big(I+ {\wh{D}}^{a} \Big), \quad \wh{D}^{a}=\text{diag}(d_k^{a})_{k\in \mathcal R},
\end{equation}
where for some $\wt \epsilon>0,$
$$\quad d_{k}^a:=\begin{cases} d_k^a, \ &\text{if } 1\le k\le r\\
\wt\epsilon, \ &\text{if } k\in \mathcal R \text{ and } k>r\end{cases}.
$$
Then all the previous proof goes through for the perturbed model as long as we replace the $\mathbf U$ and $\bm{\mathcal D}$ in \eqref{eq_defnudefnd} with
\begin{equation}\label{Depsilon0_general}
\wh{\mathbf{U}}=
\begin{pmatrix}
\mathbf{E}_{r+2} & 0 \\
0 & \mathbf{E}_s
\end{pmatrix}, \quad \wh{\bm{\mathcal{D}}}=
\begin{pmatrix}
\wh D^a(\wh D^a+1)^{-1} & 0 \\
0 & D^b(D^b+1)^{-1}
\end{pmatrix}.
\end{equation}
Note that in the proof, only the upper bound on the $d_k^a$'s were used. Moreover, the proof does not depend on the fact that $\wh{a}_i$ or $\wh{a}_j$ satisfy \eqref{eq_outlierlocation} (we only need the indices in $S$ to satisfy Assumption \ref{assum_eigenvector}). By taking $\wt\epsilon\downarrow 0$ and using continuity, we get that Proposition \ref{prop_maineigenvector} holds for general $i,j \in \{1,\cdots, N\}$. 

\vspace{2pt}
\noindent{\bf Step 2:} In the second step, we complete our proof by removing Assumption \ref{assum_eigenvector}. In fact, once we finish the proof of Step one, the second step is { relatively standard and follows the same argument} as in \cite[Section E.2]{DYaos} and \cite[Section 5.2]{Bloemendal2016}. The main idea behind is to split the discussion into several cases by considering two sets related to $S$; see Definition E.8 of \cite{DYaos}. We omit the details since our proof follows verbatim as \cite[Section E.2]{DYaos}. This completes our proof.    

\end{proof}

\subsection{Proof of Theorem \ref{thm_nonoutliereigenvector}}\label{sec_proofofnonoutliereigenvector}

In this subsection, we prove the results for the non-outlier eigenvectors. Our goal is to prove the following proposition, from which Theorem \ref{thm_nonoutliereigenvector} immediately follows. 
\begin{prop}\label{prop_outliereigenvector} Fix a small constant $\widetilde{\tau} \in (0, 1/3).$ For $\pi_a(i) \notin \mathcal{O}^+$ and $i \leq \tau N,$ where $\tau>0$ is given in Theorem \ref{thm_nonoutliereigenvector}, we have 
\begin{equation*}
\left| \langle \mathbf{e}_j, \widehat{\ub}_{\pi_a(i)} \rangle \right|^2 \prec \frac{1}{N(\kappa_i+|\widehat{a}_j-\Omega_{\beta}(E_{+})|^2)}. 
\end{equation*}
Moreover, if $\pi_a(i) \in \mathcal{O}^+$ satisfies 
\begin{equation*}
\widehat{a}_i-\Omega_{\beta}(E_{+}) \leq N^{-1/3+\widetilde{\tau}},
\end{equation*} 
we have that 
\begin{equation}\label{eq_extrabound}
\left| \langle \mathbf{e}_j, \widehat{\ub}_{\pi_a(i)} \rangle \right|^2 \prec \frac{N^{4 \widetilde{\tau}}}{N(\kappa_i+|\widehat{a}_j-\Omega_{\beta}(E_{+})|^2)}. 
\end{equation} 
Similar results hold for the right singular vectors.  
\end{prop}

The rest of the subsection is devoted to the proof of Proposition \ref{prop_outliereigenvector}. Its proof is similar to Proposition F.1 of \cite{DYaos}. We focus on explaining the main differences following the strategy outlined in Section \ref{sec_proofstartegy}. 

\begin{proof}[\bf Proof of Proposition \ref{prop_outliereigenvector}] Due to similarity, we only prove the first statement.  As mentioned earlier, the control relies on the deterministic inequality (\ref{nonoutlier_keyexpansion}). By Theorems \ref{prop_linearlocallaw}, \ref{thm_outlier}, \ref{thm_outlieereigenvector}, \ref{thm_eigenvaluesticking} and Lemmas \ref{thm_rigidity} and \ref{thm_delocalization}, for any fixed $\epsilon>0,$ we can choose a high-probability event $\Xi_3$ where (\ref{eq_proofinsidelocallaw})--(\ref{eq_proofoutlier}), (\ref{eq_stickingone})--(\ref{eq_stickingtwo}) and the following hold:
\begin{equation*}
\mathbf{1}(\Xi_3)|\widehat{\lambda}_i-\gamma_i| \leq C i^{-1/3} n^{-2/3+\epsilon/2}, \ \text{for} \ \pi_a(i) \notin \mathcal{O}^+ \ \text{and} \  i \leq \tau p. 
\end{equation*}
where we used the interlacing result Lemma \ref{lem_interlacing}. 

In what follows, we again focus on $\Xi_2$ and the discussion will be purely deterministic. Recall $\bm{e}_i$ is the natural embedding of $\mathbf{e}_i$ in $\mathbb{R}^{2N}.$ As discussed in (\ref{Depsilon0_general}), for $1 \leq j \leq \tau N,$ we may define a large set $\{1,2,3,\cdots, r\} \cup \{j\}$ to handle the general case. For simplicity, we assume that $1 \leq j \leq r$ to simplify our notations. Let $\eta_{i}>0$ be the unique solution of 
\begin{equation}\label{eq_etaidefinition}
	\im \Omega_B(\wh{\lambda}_{i}+\ii\eta_{i})=\frac{N^{6 \epsilon}}{\eta_i N},
\end{equation} 
which follows from the facts that $N^{-1+\epsilon}\im \Omega_{B}(\wh{\lambda}_{i}+\ii N^{-1+\epsilon})\lesssim N^{-1+\epsilon}$ due to Lemma \ref{prop:stabN} and the mapping $\eta\mapsto \eta\im\Omega_{B}(\wh{\lambda}_{i}+\ii\eta)$ is monotone increasing (see \cite[Lemma A.8]{DJ1}). \normalcolor According to (\ref{nonoutlier_keyexpansion}) and (\ref{eq_linearizationspectral}), fixing an $\pi_a(i) \notin \mathcal{O}^+,$ we can write 
\begin{equation}\label{eq_representationone}
\left| \langle \mathbf{e}_j, \widehat{\ub}_{\pi_a(i)} \rangle \right|^2 \leq \eta_i \im \langle \bm{e}_j, \widehat{\Gb}(z_{i})\normalcolor \bm{e}_j \rangle , \quad z_i=\widehat{\lambda}_i+\ii \eta_i.
\end{equation}
Recall (\ref{eq_defnez}) {and (\ref{eq_resolventexpansion})}. By (\ref{eq_keyexpansionev}), using the resolvent expansion \eqref{eq_resolventexpansion}, we find that 
\begin{align}\label{eq_importantexpansion}
z \langle \bm{e}_j,  \widehat{\Gb}(z)\normalcolor \bm{e}_j\rangle=\frac{1}{d_j^a}-\frac{1+d_j^a}{(d_j^a)^2} \left[ \Phi_j(z)+\Phi_j^2(z) \left( \mathcal{E}(z)+\mathcal{E}(z) \frac{1}{\mathcal{D}^{-1}+z \Ub^* \Gb(z) \Ub} \mathcal{E}(z) \right)_{jj} \right],
\end{align}
and used the abbreviated notation that
\begin{equation*}
\Phi_j(z):=\left(\frac{1}{d_j^a}-\frac{a_j}{\Omega_B(z)-a_j} \right)^{-1}.
\end{equation*}
Similar to the discussion between (F.7) and (F.15) of \cite{DYaos}, by (\ref{eq_proofinsidelocallaw}), \eqref{eq_localquadtratic}, Lemma \ref{prop:stabN} and the definition (\ref{eq_etaidefinition}), we find that 
\begin{equation*}
\min_j|\Phi_j(z)| \gg \| \mathcal{E}(z_i) \|, 
\end{equation*}
and consequently by (\ref{eq_importantexpansion}), we have 
\begin{equation}\label{eq_representationtwo}
z_i \langle \bm{e}_j,  {\widehat{\Gb}(z_i)} \bm{e}_j\rangle =\frac{\Omega_B(z_i)}{\widehat{a}_j-\Omega_B(z_i)}+\rO \left( \frac{\| \mathcal{E}(z_i) \|}{|\widehat{a}_j-\Omega_B(z_i)|^2}\right). 
\end{equation}
In order to bound the right-hand side of the above equation, we will need the following estimate. Its proof is the same as (6.10) of \cite{Bloemendal2016} or Lemma F.2 of \cite{DYaos} and we omit the details.
\begin{lem}\label{lem_keyestimate} For any fixed $\delta \in [0, 1/3-\epsilon),$ there exists a constant $c>0$ such that 
\begin{equation*}
|\widehat{a}_j-{\Omega_B(z_i)}| \geq c\left[N^{-2\delta}|\widehat{a}_j-\Omega_{\beta}(E_{+})|+\im \Omega_B(z_i) \right],
\end{equation*}
holds whenever $\widehat{\lambda}_i \in  [0, \Omega_B^{-1}( \Omega_{\beta}(E_{+})+N^{-1/3+\delta+\epsilon})).$
\end{lem}

By (\ref{eq_representationone}), (\ref{eq_representationtwo}), and fixing a suitable $\delta>0$ in Lemma \ref{lem_keyestimate}, we find that 
\begin{equation}\label{eq_lastestimate}
\left| \langle \mathbf{e}_j, \widehat{\ub}_{\pi_a(i)} \rangle \right|^2 \leq -\frac{\eta_i^2}{|z_i|^2} \re (\widehat{a}_j-\Omega_B(z_i))^{-1}-\frac{\eta_i \widehat{\lambda}_i}{|z_i|^2} \im (\widehat{a}_j-\Omega_B(z_i))^{-1}+\frac{C \eta_i \| \mathcal{E}(z_i) \|}{|\widehat{a}_j-\Omega_B(z_i)|^{2}}.
\end{equation}
It only remains to control the terms on the right-hand side of (\ref{eq_lastestimate}) one by one. For the first term, we find that 
\begin{equation*}
 \left| \frac{\eta_i^2}{|z_i|^2} \re (\widehat{a}_j-\Omega_{\beta}(E_{+}))^{-1} \right| \leq C \frac{\eta_i^2}{\im \Omega_B(z_i) } \leq C \eta_i^2 N^{1-6\epsilon} \leq C N^{-1+6 \epsilon+3 \delta},
\end{equation*}
where we used the definition of $\eta_i$ in (\ref{eq_etaidefinition}), \eqref{eq_localquadtratic} and (4) of Lemma \ref{prop:stabN}. 
The other two terms can be analyzed similarly and we refer the readers to Section F of \cite{DYaos}. {
For example, for the second term on the right-hand side of (\ref{eq_lastestimate}), we have that on the probability event $\Xi_3$
\begin{align*}
\frac{\eta_i \widehat{\lambda_i}}{|z_i|^2} \operatorname{Im} \left( \widehat{a}_j-\Omega_B(z_i) \right)^{-1} \sim \eta_i \frac{ \operatorname{Im} \Omega_B(z_i)}{|\widehat{a}_j-\Omega_B(z_i)|^2}.
\end{align*}
According to Lemma \ref{lem_keyestimate}, we find that $|\widehat{a}_j-\Omega_B(z_i)|^{-2}=\mathrm{O}(|\widehat{a}_j-\Omega_{\beta}(E_{+})|^{-2})$ by setting $\delta=0.$ Moreover, by (\ref{eq_etaidefinition}), we find that $\operatorname{Im} \Omega_B(z_i)=N^{6 \epsilon}(\eta_i N)^{-1}.$ Consequently, we see that for some constant $C>0,$ 
\begin{equation*}
\frac{\eta_i \widehat{\lambda_i}}{|z_i|^2} \operatorname{Im} \left( \widehat{a}_j-\Omega_B(z_i) \right)^{-1}  \leq C \frac{N^{-1+6 \epsilon}}{|\widehat{a}_j-\Omega_B(z_i)|}.
\end{equation*}
} Using these estimates and (\ref{eq_lastestimate}), we immediately arrive at
\begin{equation}\label{eq_llllll}
 \left| \langle \mathbf{e}_j, \widehat{\ub}_{\pi_a(i)} \rangle \right|^2 \leq CN^{-1+6 \epsilon+3 \delta}+C\frac{N^{-1+6 \epsilon}}{|\widehat{a}_j-\Omega_B(z_i)|^2}.
\end{equation}
Finally, to conclude the proof, we need to provide a lower bound for the denominator of the second term of the right-hand side of (\ref{eq_llllll}). The lower bound directly follows from Lemma \ref{lem_keyestimate} by choosing $\delta=0$. This concludes our proof. 
\end{proof}
\subsection{Proof of Theorem \ref{thm_outlieereigenvector}}\label{sec_prove36withoutconstrain}

In this subsection, we complete the proof of Theorem \ref{thm_outlieereigenvector} by removing (\ref{condition_nonoverone}) from the proof in Section \ref{sec_prove36withconstrain} using the estimate (\ref{eq_extrabound}). The proof is similar to the discussion in Section F of \cite{DYaos} and we only provide the main points.

As discussed earlier, it suffices to prove Proposition \ref{prop_maineigenvector} without imposing (\ref{condition_nonoverone}). Fix a constant $\epsilon>0,$ then it is easy to check that there exists some $x_0 \in [1, r+s+1]$  so that there is no $\wh{a}_{k}$ between $\Omega_{\beta}(E_{+})+x_0N^{-1/3+\epsilon}$ and $\Omega_{\beta}(E_{+})+(x_0+1)N^{-1/3+\epsilon}.$ Thus, following the ideas of \cite[Section 6.2]{Bloemendal2016} or \cite[Section F]{DYaos}, we may split $S=S_0 \cup S_1$ so that $\widehat{a}_k \leq \Omega_{\beta}(E_{+})+x_0 N^{-1/3+\epsilon}$ for $\pi_a(k) \in S_0,$ and $\widehat{a}_k>\Omega_{\beta}(E_{+})+(x_0+1) N^{-1/3+\epsilon}$ for $\pi_a(k) \in S_1.$ Without loss of generality, we assume that $S_0 \neq \emptyset$, otherwise the proof is complete. 

Based on the above discussion, the actual proof is divided into six cases according to which of the sets {$S_0,$ $S_1$ or $S^c$} the indices $\pi_a(i)$ and $\pi_a(j)$ belong to. Due to similarity, we restrict our attention to few specific cases. Other cases can be handled in a similar fashion.  The proof relies on the decomposition
\begin{equation}\label{eq_setdecomposition}
\langle \mathbf{e}_i, \mathcal{P}_S \mathbf{e}_j \rangle=\langle \mathbf{e}_i, \mathcal{P}_{S_0} \mathbf{e}_j \rangle+\langle \mathbf{e}_i, \mathcal{P}_{S_1} \mathbf{e}_j \rangle. 
\end{equation}
When $\pi_a(i), \pi_a(j) \in S_0,$ applying {the Cauchy-Schwarz inequality} and the estimate (\ref{eq_extrabound}) to the first term of the right-hand side of  (\ref{eq_setdecomposition}), and Proposition \ref{prop_maineigenvector} to the second term of it, we obtain that 
\begin{align*}
\left| \langle \mathbf{e}_i, \mathcal{P}_S \mathbf{e}_j \rangle-\delta_{ij} {\mathbf{1}}(\pi_a(i) \in S) \wh{a}_i \frac{(\Omega_B^{-1})'(\wh{a}_i)}{\Omega_B^{-1}(\wh{a}_i)} \right| & \prec \delta_{ij} \Delta_1(\widehat{a}_i)^2+\frac{N^{4 \epsilon}}{N \Delta_1(\widehat{a}_i)^2 \Delta_1(\widehat{a}_j)^2}+\frac{1}{N \delta_{\pi_a(i)}(S_1)\delta_{\pi_a(j)}(S_1)} \\
& =\rO \left( \frac{N^{4 \epsilon}}{N \Delta_1(\widehat{a}_i)^2 \Delta_1(\widehat{a}_j)^2} \right),
\end{align*}
where in the second step we used the fact that $\Delta_1(\widehat{a}_{i(j)})^2=\rO(N^{-1/3+\epsilon})=\rO( \delta_{\pi_a(i(j))}(S_1)).$ Similarly, we can obtain the results when $\pi_a(i), \pi_a(j) \in S_1$ using Propositions \ref{prop_outliereigenvector} and \ref{prop_maineigenvector}. Moreover, when $\pi_a(i) \in S_0, \pi_a(j) \in S_1,$ the proof can be divided into two steps. In the first step, we prove the results under Assumption \ref{assum_eigenvector} for some constant $0<\tau_2<\epsilon$. In the second step we remove this assumption as in the proof of Proposition \ref{prop_maineigenvector}. We only highlight the first step.  Applying {the Cauchy-Schwarz inequality} and Proposition \ref{prop_outliereigenvector} to the first term of the RHS of (\ref{eq_setdecomposition}) , and applying Proposition \ref{prop_maineigenvector} to the second term, we get that 
\begin{align*}
\left|  \langle \mathbf{e}_i, \mathcal{P}_S \mathbf{e}_j \rangle \right| & \prec \frac{N^{4 \epsilon}}{N \Delta_1(\widehat{a}_i)^2 \Delta_1(\widehat{a}_j)^2}+\frac{1}{N\delta_{\pi_a(i)}(S_1)} \left( \frac{1}{\delta_{\pi_j(S_1)}}+\frac{1}{\Delta_1(\widehat{a}_j)^2} \right)+\frac{\Delta_1(\widehat{a}_j)}{\sqrt{N} \delta^a_{\pi_a(i), \pi_a(j)}} \\
& \prec N^{4 \epsilon}\left[ \frac{1}{N\Delta_1(\widehat{a}_i)^2 } \left( \frac{1}{\delta_{\pi_j(S_1)}}+\frac{1}{\Delta_1(\widehat{a}_j)^2} \right)+\frac{1}{\sqrt{N} \Delta_1(\widehat{a}_i) \Delta_1(\widehat{a}_j)} \right],
\end{align*}
where we used the facts that $\Delta_1(\widehat{a}_i)^2=\rO(\delta_{\pi_a(i)}(S_1))=\rO(\Delta_1(\widehat{a}_j)^2)=\rO(\delta^a_{\pi_a(i), \pi_a(j)})$ and $\delta_{\pi_a(j)}(S) \wedge |\widehat{a}_i-\Omega_{\beta}(E_{+})|=\rO(\delta_{\pi_a(j)}(S_1)).$ The other cases can be discussed similarly and we refer the readers to Section F of \cite{DYaos} for more details. This concludes our proof of Theorem \ref{thm_outlieereigenvector}.

\section*{Acknowledgment}  

{The authors would like to thank the editor, the associated editor and two anonymous referees for their many
critical suggestions which have significantly improved the paper.} The authors are also grateful to Zhigang Bao and Ji Oon Lee for many helpful discussions. The first author also wants to thank Hari Bercovici for many useful comments. The first author is partially supported by NSF DMS-2113489 and the second author is supported by ERC Advanced Grant "RMTBeyond" No.~101020331.

\bibliographystyle{myjmva}
\bibliography{multi,Ref}

\end{document}